\documentclass{article}
\usepackage[english]{babel}
\usepackage{amsmath}
\usepackage{amsthm}
\usepackage{amssymb}
\usepackage{amscd, enumerate}
\usepackage[latin1]{inputenc}

\newtheorem{teor}{Theorem}[section]
\newtheorem{defi}{Definition}
\newtheorem{lema}[teor]{Lemma}
\newtheorem{prop}[teor]{Proposition}
\newtheorem{cor}[teor]{Corollary}
\newtheorem{rem}[teor]{Remark}

\newtheorem{ejems}[teor]{Examples}

\begin{document}

\title{Pseudo-Frobenius graded algebras with enough idempotents}
\author{Estefanía Andreu Juan\\ eaj1@um.es
\and Manuel Saorín\\msaorinc@um.es}

\date{}
\maketitle

\begin{center}

Departmento de Matem\'aticas

Universidad de Murcia

Campus de Espinardo, 30100 Murcia

Spain \end{center}

\vspace{1cm}

\begin{center}
Dedicated to Alberto Facchini in his 60th birthday
\end{center}

\vspace{1cm}

\begin{abstract}
We introduce and study the notion of pseudo-Frobenius graded algebra with enough idempotents, showing that it follows the pattern of  the classical concept of pseudo-Frobenius (PF) and Quasi-Frobenius (QF) ring, in particular finite dimensional self-injective algebras,  as studied by Nakayama,  Morita, Faith, Tachikawa,  etc. We show that such an algebra is characterized by the existence of a graded Nakayama form. Moreover, we prove that the pseudo-Frobenius property is preserved and reflected by covering functors, a fact which makes the concept useful in Representation Theory.
\end{abstract}

\vspace{0.4cm}


\vspace{0.4cm}

\noindent Classification Code: 16Gxx ; Representation theory of
rings and algebras.

\section{Introduction}\oddsidemargin=
3cm \evensidemargin=1.5cm \textheight=24cm \voffset=-1.5cm

In Ring Theory and Representation Theory of Algebras it has been a classical and natural process to abstract the properties of group algebras over a field and define new contexts where many of those properties still hold and to which, as a consequence, results from group algebras are easily transfered. In this vein, the concept of finite dimensional self-injective algebra already appeared in the work of Nakayama (\cite{Nak1}, \cite{Nak2}) and is nowdays a common object of study. On the ring-theoretic side,  Quasi-Frobenius (QF) rings  are the 'field-free' equivalent of finite dimensional self-injective algebras and were intensively studied, among others, by Tachikawa (see \cite{T}). The most important feature of a Quasi-Frobenius ring is that its  (left or right) module category is a Frobenius category, i.e., the classes of injective and projective modules coincide. Going further in this direction, Osofsky \cite{O} introduced a class of rings, later called
 pseudo-Frobenius (PF). Left PF rings are those rings $R$ for which the regular module ${}_RR$ is an injective cogenerator. Using known characterizations of such rings (see \cite{F}[Theorem 24.32]), it is not hard to prove that these are precisely the rings for which the classes of finitely generated projective and finitely cogenerated injective left modules coincide.

When dealing with split basic finite dimensional algebras over a field $K$, one of the most important tools in Representation Theory is that of coverings of quivers with relations and, more generally, covering functors of small $K$-categories. Roughly speaking, from such an original (associative unital) finite dimensional algebra $\Lambda$, which can be interpreted as $K$-category whose objects are the elements of a complete family of primitive idempotents,  one constructs, by using covering techniques, a new small $K$-category or, equivalently, an algebra with enough idempotents $\tilde{\Lambda}$ with a canonical functor $\tilde{\Lambda}\longrightarrow\Lambda$. Alhough one generally loses the unital condition of the algebra in the process, the module category over $\tilde{\Lambda}$ is  much simpler  than the one over $\Lambda$ and the two module categories are related via well-behaved functors. The process has been fundamental to tackle classification of representation-finite algebras (see \cite{Ri}, \cite{G2}, \cite{BG}).

Motivated by the power of coverings to deal with representation-theoretic problems and by our interest in studying the conditions of symmetry, periodicity and Calabi-Yau of finite dimensional mesh algebras (see \cite{AS}),  we introduce and study in this paper the notion of pseudo-Frobenius graded algebra with enough idempotents. We show that such an algebra shares a lot of features of finite dimensional self-injective algebras. In particular, it can be characterized via the existence of a graded Nakayama form and it affords a Nakayama automorphism which is fundamental for its study. What is even more important, we show that, under fairly general hypotheses, the pseudo-Frobenius condition of a graded algebra with enough idempotents is preserved and reflected via covering functors and the associated graded Nakayama form and Nakayama automorphism move rather well from a given algebra to its covering an viceversa. This later fact has been the key point for the results obtained in \cite{AS}.

In this paper the term 'algebra' will mean always an associative algebra over a ground field $K$ and all gradings on such algebra are $H$-gradings, where $H$ is an abelian group with additive notation. Both $K$ and $H$ will be fixed in the rest of the paper. This is  enough for our purposes, although some proofs and results might be adapted to more general situations.

The organization of the paper goes as follows. In section 2 we give the definition and basic properties of a graded algebra with enough idempotents, showing that it can be interpreted as a small graded $K$-category and viceversa. We characterize the graded Jacobson radical of such an algebra and define the concept of weakly basic graded algebra with enough idempotents, which is the analogue of a basic algebra in the ungraded unital setting. In section 3, as way of introduction of the concepts,  we prove Theorem \ref{teor.graded pseudo-Frobenius algebras} which characterizes, in the weakly basic case,  the
 pseudo-Frobenius graded algebras and shows that Quasi-Frobenius graded algebras with enough idempotents are always pseudo-Frobenius, the converse being true in the locally Noetherian case (corollary \ref{cor.Frobenius=simple socle}). The new concept is given via a graded Nakayama form, which comes with an associated Nakayama permutation of a weakly basic set of idempotents in the algebra and with a degree map from this set to $H$. It is shown in Proposition \ref{prop.uniqueness of Nakayama permutation} that this permutation depends only on the algebra and not on the Nakayama form itself, while the analogous fact is not true for the degree map.
In  Proposition 3.7 we give an explicit way of constructing a graded Nakayama form of a split pseudo-Frobenius graded algebra. We end the section by showing that, when  $A$ is a graded pseudo-Frobenius algebra given by a $\mathbb{Z}$-graded quiver with homogenous relations in the standard way,  one can choose a graded Nakayama form for $A$ with constant degree map. In the final section 4, we adapt to the graded situation the theory of coverings, using the recent approach to the topic of \cite{CM} and \cite{Asa2}, which is more general than the original one. The main result of the section (Theorem \ref{teor.lifting PF in split case}) states that, in the weakly basic locally bounded case and when the involved group of graded automorphisms acts freely on objects, the associated covering functor preserves and reflects the pseudo-Frobenius condition. As a consequence (see Corollary \ref{cor.covering of selfinjective algebra}), one get that the universal cover of a split basic finite dimensional self-injective algebra $\Lambda$ is a (trivially graded) pseudo-Frobenius algebra and that any finite dimensional Galois covering of $\Lambda$ is also a self-injective algebra.

\section{Graded algebras with enough idempotents}

\subsection{Definition and basic properties} \label{subsect. gr-algebras enough-idempt}

 Recall that  an associative
algebra $A$ is said to be an  {\bf algebra with enough idempotents},
when there is a  family $(e_i)_{i\in I}$ of nonzero orthogonal
idempotents such that  $\oplus_{i\in I}e_iA=A=\oplus_{i\in I}Ae_i$.
Any such family $(e_i)_{i\in I}$ will  be called a {\bf
distinguished family}. From now on in this paper $A$ is an algebra
with enough idempotents on which  we fix a distinguished family of
orthogonal idempotents.

All considered (left or right) $A$-modules are supposed to be
unital. For a left (resp. right) $A$-module $M$, that means that
$AM=M$ (resp. $MA=M$) or, equivalently, that we have a decomposition $M=\oplus_{i\in I}e_iM$
(resp. $M=\oplus_{i\in I}Me_i$) as $K$-vector spaces. We denote by $A-\text{Mod}$ and
$\text{Mod}-A$ the categories of left and right $A$-modules,
respectively.

The enveloping algebra of $A$ is the algebra $A^e=A\otimes A^{op}$. Here and in the rest of the paper, unadorned tensor products are considered over $K$. The algebra $A^e$ is also an algebra with enough
idempotents and, when $a,b\in A$,  we will denote by $a\otimes b^o$ the
corresponding element of $A^e$. The distinguished family of orthogonal idempotents with
which we will work is the family $(e_i\otimes e_j^o)_{(i,j)\in
I\times I}$. A left $A^e$-module $M$ will be identified with an
$A$-bimodule by putting $axb=(a\otimes b^o)x$, for all $x\in M$ and
$a,b\in A$. Similarly, a right $A^e$-module is identified with an
$A$-bimodule by putting $axb=x(b\otimes a^o)$, for all $x\in M$ and
$a,b\in A$. In this way, we identify the three categories
$A^e-\text{Mod}$, $\text{Mod}-A^e$ and $A-\text{Mod}-A$, where the
last one is the category of unitary $A$-bimodules, which we will
simply name 'bimodules'.

 Let $H$ be an abelian group with additive notation. An {\bf $H$-graded algebra with enough
idempotents} will be an algebra with enough idempotents $A$,
together with an $H$-grading $A=\oplus_{h\in H}A_h$, such that one
can choose a distinguished family of orthogonal idempotents which
are homogeneous of degree $0$. Such a family $(e_i)_{i\in I}$ will
be fixed from now on. We will denote by $A-Gr$ (resp.  $Gr-A$) the
category ($H$-)graded (always unital) left (resp. right) modules,
where the morphisms are the graded homomorphisms of degree $0$. A
{\bf locally finite dimensional left (resp. right) graded
$A$-module} is a graded module $M=\oplus_{h\in H}M_h$ such that, for
each $i\in I$ and each $h\in H$, the vector space $e_iM_h$ (resp.
$M_he_i$) is finite dimensional. Note that  the definition does not
depend on the distinguished family $(e_i)$. We will denote by
$A-lfdgr$ and $lfdgr-A$ the categories of left and right locally
finite dimensional graded modules.

 Given a graded left
$A$-module $M$, we denote by $D(M)$ the subspace of the vector space
$\text{Hom}_K(M,K)$ consisting of the linear forms
$f:M\longrightarrow K$ such that $f(e_iM_h)=0$, for all but finitely
many $(i,h)\in I\times H$. The $K$-vector space $D(M)$ has a
canonical structure of graded right $A$-module given as follows. The
multiplication $D(M)\times A\longrightarrow D(M)$ takes
$(f,a)\rightsquigarrow fa$, where $(fa)(x)=f(ax)$ for all $x\in M$.
Note that then one has $fe_i=0$, for all but finitely many $i\in I$,
and $f=\sum_{i\in I}fe_i$. Therefore $D(M)$ is unital. On the other
hand, if we put $D(M)_h:=\{f\in D(M):$ $f(M_k)=0\text{, for all
$k\in H\setminus\{-h\}$}\}$, we get a decomposition
$D(M)=\oplus_{h\in H}D(M)_h$ which makes $D(M)$ into a graded right
$A$-modules. Note that $D(M)_he_i$ can be identified with
$\text{Hom}_K(e_iM_{-h},K)$, for all $(i,h)\in I\times H$. We will
call $D(M)$ the {\bf dual graded module} of $M$.

Recall that if $M$ is a graded $A$-module and $k\in H$ is any
element, then we get a graded module $M[k]$ having the same
underlying ungraded $A$-module as $M$, but where $M[k]_h=M_{k+h}$
for each $h\in H$. If $M$ and $N$ are graded left $A$-modules, then
$\text{HOM}_{A}(M,N):=\oplus_{h\in H}\text{Hom}_{A-Gr}(M,N[h])$
 has a structure of graded $K$-vector space, where the homogeneous
 component of degree $h$ is
 $\text{HOM}_{A}(M,N)_h:=\text{Hom}_{A-Gr}(M,N[h])$, i.e.,
 $\text{HOM}_{A}(M,N)_h$ consists of the graded homomorphisms $M\longrightarrow N$ of
 degree $h$.
The following is an analogue of  classical results for modules over associative
rings with unit, whose proof can be easily adapted. It is left to
the reader.

\begin{prop} \label{prop.the functor D over algebras with s.i.}
The assignment $M\rightsquigarrow D(M)$ extends to an exact
contravariant $K$-linear functor $D:A-Gr\longrightarrow Gr-A$ (resp.
$D:Gr-A\longrightarrow A-Gr$) satisfying the following properties:

\begin{enumerate}
\item The maps $\sigma_M:M\longrightarrow D^2(M):=(D\circ D)(M)$,
where $\sigma_M(m)(f)=f(m)$ for all $m\in M$ and $f\in D(M)$, are
all injective and give a natural transformation $\sigma
:1_{A-Gr}\longrightarrow D^2:=D\circ D$ (resp. $\sigma
:1_{Gr-A}\longrightarrow D^2:=D\circ D$)

\item  If $M$ is locally finite
dimensional then $\sigma_M$ is an isomorphism

\item The restrictions of $D$ to the subcategories of locally finite
dimensional graded $A$-modules define mutually inverse dualities
$D:A-lfdgr\stackrel{\cong^{op}}{\longleftrightarrow}lfdgr-A:D$.

\item If $M$ and $N$ are a left and a right graded $A$-module, respectively, then there is an isomorphism of
graded $K$-vector spaces

\begin{center}
$\eta_{M,N}:\text{HOM}_{A}(M,D(N))\longrightarrow D(N\otimes_AM)$,
\end{center}
which is natural on both variables.
\end{enumerate}
\end{prop}

When $A=\oplus_{h\in H}A_h$ and $B=\oplus_{h\in H}B_h$ are graded
algebras with enough idempotents, the tensor algebra $A\otimes B$
inherits a structure of graded $H$-algebra, where $(A\otimes
B)_h=\oplus_{s+t=h}A_s\otimes B_t$. In particular, this applies to
the enveloping algebra $A^e$ and, as in the ungraded case, we will
identify the categories $A^e-Gr$ (resp. $Gr-A^e$) and $A-Gr-A$ of
graded left (resp. right) $A^e$-modules and graded $A$-bimodules.

\begin{rem} \label{rem.dual of bimodules and modules}
If $M$ is a graded $A$-bimodule and  we denote  by $D(_AM)$,
$D(M_A)$ and $D(_AM_A)$, respectively, the duals of $M$ as a left
module, right module or bimodule, then $D(_AM_A)=D(_AM)\cap D(M_A)$
and, in general, $D(_AM)$ and $D(M_A)$ need not be the same vector
subspace of $\text{Hom}_K(M,K)$. They are equal if the following two
properties hold:

\begin{enumerate}
\item For each  $(i,h)\in I\times H$, there are only finitely many $j\in I$ such
that $e_iM_he_j\neq 0$ \item For each $(i,h)\in I\times H$, there
are only finitely many $j\in I$ such that $e_jM_he_i\neq 0$.
\end{enumerate}
\end{rem}

\begin{rem} \label{rem.trivially graded algebra}
 When $H=0$, we have $A-Gr=A-\text{Mod}$ and $D(M)=\{f:M\longrightarrow K:$ $f(e_iM)=0\text{, for almost all }i\in
I\}$.
\end{rem}

\begin{defi} \label{defi.locally f.d. graded algebra}
Let $A=\oplus_{h\in H}A_h$ be a graded algebra with enough
idempotents. It will be called {\bf locally finite dimensional} when
the regular bimodule $_AA_A$ is locally finite dimensional, i.e.,
when $e_iA_he_j$ is finite dimensional, for all $(i,j,h)\in I\times
I\times H$. Such a graded algebra $A$ will be called {\bf  graded
locally bounded} when the following two conditions hold:

\begin{enumerate}
\item For each $(i,h)\in I\times H$, the set $I^{(i,h)}=\{j\in I:$ $e_iA_he_j\neq
0\}$ is finite
\item For each $(i,h)\in I\times H$, the set $I_{(i,h)}=\{j\in I:$ $e_jA_he_i\neq
0\}$ is finite.
\end{enumerate}
\end{defi}

\begin{rem} \label{rem.the concepts for H=0}
For $H=0$, the just defined concepts are the familiar ones of
locally finite dimensional and locally bounded, introduced in the
language of $K$-categories by Gabriel and collaborators (see, e.g.,
\cite{BG}).
\end{rem}

\subsection{Graded algebras with enough idempotents versus graded $K$-categories}
\label{sec:Graded algebras vs graded categories}

In this subsection we remind the reader that graded algebras with
enough idempotents can be looked at as small graded $K$-categories,
and viceversa.

A  category $\mathcal{C}$  is a $K$-category if $\mathcal{C}(X,Y)$
is a $K$-vector space, for all objects $X,Y$, and the composition
map $\mathcal{C}(Y,Z)\times \mathcal{C}(X,Y)\rightarrow
\mathcal{C}(X,Z)$ is $K$-bilinear, for all
$X,Y,Z\in\text{Ob}(\mathcal{C})$. If now $H$ is a fixed additive
abelian group, then $\mathcal{C}$ is an \emph{$(H-)$ graded
$K$-category} if $\mathcal{C}(X,Y)= \oplus_{h\in
H}\mathcal{C}_h(X,Y)$ is a graded $K$-vector space, for all
$X,Y\in Obj(\mathcal{C})$, and  the composition map restricts to a
($K$-bilinear) map

\begin{center}
  $\mathcal{C}_h(Y,Z)\times \mathcal{C}_k(X,Y)\rightarrow \mathcal{C}_{h+k}(X,Z)$
\end{center}

\noindent for any $h,k\in H$. There is an obvious definition of
\emph{graded functor (of degree zero)} between graded $K$-categories
whose formulation we leave to the reader.

The prototypical example of graded $K$-category is $(K,H)-GR=K-GR$.
Its objects are the $H$-graded $K$-vector spaces and we define
$\text{Hom}_{K-GR}(V,W):=HOM_K(V,W)=\oplus_{h\in H}\text{Hom}_{K-Gr}(V,W[h])$
(see the paragraph preceding Proposition \ref{prop.the functor D over algebras with s.i.}).

If $A=\oplus_{h\in H}A_h$ is a graded algebra with enough
idempotents, on which we fix a distinguished family $(e_i)_{i\in I}$
of orthogonal idempotents of degree zero, then we can look at it as
a small graded $K$-category. Indeed we put $\text{Ob}(A)=I$,
$A(i,j)=e_iAe_j$ and take as composition map $e_jAe_k\times
e_iAe_j\longrightarrow e_iAe_k$ the antimultiplication: $b\circ
a:=ab$.

Conversely, if $\mathcal{C}$ is a small graded $K$-category then
$R=\oplus_{X,Y\in\text{Ob}(\mathcal{C})}\mathcal{C}(X,Y)$ is a graded $K$-algebra
with enough idempotents, where the family of identity maps
$(1_X)_{X\in\text{Ob}(\mathcal{C})}$ is a distinguished family of
homogeneous elements of degree zero. We will call $R$ the
\emph{functor algebra associated to $\mathcal{C}$}. Let
$GrFun(\mathcal{C}, K-GR)$ denote the category of graded $K$-linear
covariant functors, with morphisms the $K$-linear natural
transformations. To each object $F$ in this category, we canonically
associate a graded left $R$-module $\mathcal{M}(F)$ as follows. The
underlying graded $K$-vector space is
$\mathcal{M}(F)=\oplus_{C\in\text{Ob}(\mathcal{C})}F(C)$. If $f\in
1_YR1_X=\mathcal{C}(X,Y)$ and $z\in F(Z)$, then we define $f\cdot
z=\delta_{XZ}F(f)(x)$, where $\delta_{XZ}$ is the Kronecker symbol.
Note that if $x\in F(X)$ then $f\cdot x$ is an element of $F(Y)$, and if $f$ and $x$ are
homogeneous elements, then $f\cdot x$ is homogeneous of degree
$\text{deg}(f)+\text{deg}(x)$.

Conversely, given a graded left $R$-module $M$, we can associate to
it a graded functor $F_M:\mathcal{C}\longrightarrow K-GR$ as
follows. We define $F_M(X)=1_XM$, for each
$X\in\text{Ob}(\mathcal{C})$, and if $f\in\mathcal{C}(X,Y)=1_YR1_X$
is any morphism, then $F_M(f):F_M(X)\longrightarrow F_M(Y)$ maps
$x\rightsquigarrow fx$.

Given an object $X$ of the graded $K$-category $\mathcal{C}$, the
associated \emph{representable functor} is the functor
$\text{Hom}_\mathcal{C}(X,-):\mathcal{C}\longrightarrow K-GR$ which
takes $Y\rightsquigarrow\mathcal{C}(X,Y)$, for each
$Y\in\text{Ob}(\mathcal{C})$. With an easy adaptation of the proof
in the  ungraded case (see, e.g., \cite{G0}[Proposition II.2]), we
get:

\begin{prop}\label{prop.algebra de funtores de Gabriel}
Let $\mathcal{C}$ be a small ($H$-)graded $K$- category
 and let $R$ be its associated functor algebra. Then the assignments
 $F\rightsquigarrow\mathcal{M}(F)$ and $M\rightsquigarrow F_M$
 extend to mutually quasi-inverse
 equivalences of categories
 $\text{GrFun}(\mathcal{C}^{op},K-Gr)\stackrel{\cong}{\longleftrightarrow}R-Gr$.
 These equivalences restrict to mutually quasi-inverse equivalences
 $\text{GrFun}(\mathcal{C}^{op},K-lfdGR)\stackrel{\cong}{\longleftrightarrow}R-lfdgr$,
 where $K-lfdGR$ denotes the full graded subcategory of $K-GR$
 consisting of the locally finite dimensional graded $K$-vector
 spaces.

 These equivalences identify the finitely generated projective
 $R$-modules with the direct summands  of representable functors.
\end{prop}

Due to the contents of the paper, we will freely move from the
language of graded algebras with enough idempotents to that of small
graded $K$-categories and viceversa. In particular, given graded
algebras with enough idempotents $A$ and $B$, we will say that
$F:A\longrightarrow B$ is a graded functor if it so when we
interpret $A$ and $B$ as small graded $K$-categories.

\subsection{Weakly basic graded algebras with enough idempotents}

Before introducing the main concept of this subsection, we give the following useful result. Recall that, when $B=\oplus_{h\in H}B_h$ is a graded unital algebra, the intersection of its maximal graded left ideals coincides with the intersection of its maximal graded right ideals (see \cite{NVO}[Lemma I.7.4]). We denote by $J^{gr}(B)$ this ideal.

\begin{prop} \label{prop.graded Jacobson radical}
Let $A=\oplus_{h\in H}A_h$ be a graded algebra with enough idempotents, let $(e_i)_{i\in I}$ be a distinguished family of primitive idempotents and, for each finite subset $F\subseteq I$,  let us put $e_F=\sum_{i\in F}e_i$.
For a graded vector subspace $J=\oplus_{i,j}e_iJe_j$ be  of $A$, the following assertions are equivalent:

\begin{enumerate}
\item $J$ is the intersection of all maximal graded left ideals of $A$.
\item $J$ is the intersection of all maximal graded right ideals of $A$.
\item A homogeneous element $x\in e_iA_he_j$ is in $J$ if, and only if, the element $e_j-yx$ is invertible in $e_jA_0e_j$, for each $y\in e_jA_{-h}e_i$.
\item A homogeneous element $x\in e_iA_he_j$ is in $J$ if, and only if, the element $e_i-xy$ is invertible in $e_iA_0e_i$, for each $y\in e_jA_{-h}e_i$.
\item $J=\bigcup_{F\subseteq I, F\text{ finite}}J^{gr}(e_FAe_F)$.
\end{enumerate}
\end{prop}
\begin{proof}
We denote by $J(k)$ the graded vector subspace of $A$ given by assertion $k$, for $k=1,2,3,4,5$. Note that $J(k)$ is a graded left (resp. right) ideal of $A$ for $k=1,3$ (resp. $k=2,4$) and it is a two-sided ideal for $k=5$ by the previous comments. We shall prove that $J(1)=J(3)=J(5)$ and, by left-right symmetry, we will get also that $J(2)=J(4)=J(5)$.

\emph{First step: J(1)=J(3). } Put $J(1)=J$ and $J(3)=J'$ for simplicity.
 Note that if $\mathbf{m}_j$ is a maximal ideal of $e_jA_0e_j$, then $\mathbf{a}=A\mathbf{m}_j\oplus (\oplus_{i\neq j}Ae_i)$ is a graded left ideal of $A$ which does not contain $e_j$. We choose a maximal graded left ideal $\mathbf{m}$ of $A$ such that $\mathbf{a}\subseteq\mathbf{m}$. We then get inclusions $\mathbf{m}_j\subseteq\mathbf{a}\cap e_jA_0e_j\subseteq\mathbf{m}\cap e_jA_0e_j\subsetneq e_jA_0e_j$, where the last inclusion is strict since $e_j\not\in\mathbf{m}$. The maximality of $\mathbf{m}_j$ then gives that $\mathbf{m}_j=\mathbf{m}\cap e_jA_0e_j$. We then get that $e_jJ_0e_j=J\cap e_jA_0e_j$ is contained in the Jacobson radical $J(e_jA_0e_j)$ of $e_jA_0e_j$. Suppose that $x\in J\cap e_iA_he_j$ and  $y\in e_jA_{-h}e_i$. Then $yx\in e_jJ_0e_j\subseteq J(e_jA_0e_j)$. This implies that $e_j-yx$ is invertible in $e_jA_0e_j$, and so   $J\subseteq J'$.

Conversely,  take $x\in e_iJ'_he_j=J'\cap e_iA_he_j$ and suppose that $x\not\in J$. We then have a maximal graded left ideal $\mathbf{m}$ of $A$ such that $x\not\in\mathbf{m}$, so that $\mathbf{m}+Ax=A$. In particular, we get $e_j=z+yx$, for some $z\in\mathbf{m}\cap e_jA_0e_j$ and some $y\in e_jA_{-h}e_i$. It follows that $z$ is invertible in $e_jA_0e_j$, which in turn implies that $e_j\in\mathbf{m}$. But then $x=xe_j\in\mathbf{m}$, which is a contradiction.

\vspace*{0.3cm}

\emph{Second step: J(1)=J(5).} Put $J=J(1)$ and $J''=J(5)$. Let $\mathbf{m}_F$ be a maximal graded left ideal of $e_FAe_F$. Then $\mathbf{a}=(\oplus_{i\in I\setminus F}Ae_i)\oplus A\mathbf{m}_F$ is a graded left ideal of $A$ which does not contain $e_F$. Arguing as in the first step, we get a maximal graded left ideal $\mathbf{m}$ of $A$ such that $\mathbf{m}\cap e_FAe_F=\mathbf{m}_F$. We then get $J\cap e_FAe_F\subseteq J^{gr}(e_FAe_F)$, and so $J=\bigcup_{F\subseteq I, F\textbf{ finite}}J\cap e_FAe_F\subseteq\bigcup_{F\subseteq I, F\textbf{ finite}}J^{gr}(e_FAe_F)=J''$.

Note that if $F\subset F'$ are finite subsets of $I$, then, by the properties of graded unital algebras and the fact that $e_FAe_F=e_F(e_{F'}Ae_{F'})e_F$, we have that $J^{gr}(e_FAe_F)=e_FJ^{gr}(e_{F'}Ae_{F'})e_F\subset J^{gr}(e_{F'}Ae_{F'})$.
Let now $x\in e_iJ_h''e_j$ be a homogeneous element in $J''$. We can then choose a finite subset $F\subseteq I$ such that $i,j\in F$ and $x\in J^{gr}(e_FAe_F)$. By replacing $A$ and $J$ by $e_FAe_F$ and $J^{gr}(e_FAe_F)$, respectively, the first step gives that,  for each $y\in e_jA_{-h}e_i$, the element $e_j-yx$ is invertible in $e_j(e_FAe_F)e_j=e_jAe_j$. That is, we have that $x\in J(3)=J$, so that the inclusion $J''\subseteq J$ also holds.
\end{proof}

The graded two-sided ideal $J$ given by last proposition will be called the \emph{graded Jacobson radical} of $A$ and will be denoted by $J^{gr}(A)$.

\begin{defi} \label{def.basic split graded algebra}
A locally finite dimensional graded algebra with enough idempotents
$A=\oplus_{h\in H}A_h$ will be called \emph{weakly basic} when it
has a distinguished family $(e_i)_{i\in I}$ of orthogonal
homogeneous idempotents of degree $0$ such that:

\begin{enumerate}

\item   $e_iA_0e_i$ is a local algebra, for each $i\in I$ \item
$e_iAe_j$ is contained in the graded Jacobson radical
$J^{gr}(A)$, for all $i,j\in I$, $i\neq j$.
\end{enumerate}
It will be called \emph{basic} when, in addition,
$e_iA_he_i\subseteq J^{gr}(A)$, for all $i\in I$ and $h\in
H\setminus\{0\}$.

We will use also the term '(weakly) basic' to denote any
distinguished family $(e_i)_{i\in I}$ of ortho\-go\-nal idempotents
satisfying the above conditions.

A weakly basic graded algebra with enough idempotents will be called
\emph{split} when  $e_iA_0e_i/e_iJ(A_0)e_i$ $\cong K$, for each $i\in
I$.

\end{defi}

\begin{prop} \label{prop.basic graded algebra}
Let $A=\oplus_{h\in H}A_h$ be a weakly basic graded algebra with enough idempotents and let $(e_i)$ be a
weakly basic distinguished family of orthogonal idempotents. The
following assertions hold:

\begin{enumerate}
\item $J^{gr}(A)_0$ is the Jacobson radical of $A_0$. \item Each indecomposable finitely generated projective graded left $A$-module
is isomorphic to $Ae_i[h]$, for some $(i,h)\in I\times H$. Moreover,
if  $Ae_i[h]$ and $Ae_j[k]$ are isomorphism in $A-Gr$, then $i=j$
and, in case $A$ is basic, also $h=k$.

\item Each finitely generated
projective graded left $A$-module is a finite direct sum of graded
modules of the form $Ae_i[h]$, with $(i,h)\in I\times H$

\item Each finitely generated  graded left $A$-module has
a projective cover in the category $A-Gr$ \item Each finitely
generated projective graded left $A$-module is the projective cover
of a finite direct sum of graded-simple modules (=simple objects of
the category $A-Gr$).
\end{enumerate}
Moreover, the left-right symmetric versions of these assertions also
hold.
\end{prop}
\begin{proof}
 1) Put $J=J^{gr}(A)$. We need to prove that $e_iJ_0e_j=e_iJ(A_0)e_j$, for all $i,j\in I$. Let $x\in e_iA_0e_j$ be any element. By Proposition \ref{prop.graded Jacobson radical}, we have that $x\in e_iJ_0e_j$ if, and only if, $e_j-yx$ is invertible in $e_iA_0e_i$, for all $y\in e_jA_0e_i$. On the other hand, if we consider $A_0$ as a trivially graded algebra and we apply to it that same proposition, we get that this last condition is equivalent to saying that $x\in e_iJ^{gr}(A_0)e_j=e_iJ(A_0)e_j$.

The proof of the remaining assertions is  entirely similar to the
one for semiperfect (ungraded) associative algebras with unit (see,
e.g., \cite{Kasch}) and here we only summarize the adaptation,
leaving the details to the reader. For assertion 2), suppose that
there is an isomorphism $f:Ae_i[h]\stackrel{\cong}{\longrightarrow}
Ae_j[k]$ in $A-Gr$, with $(i,h),(j,k)\in I\times H$. The map  $\rho
:e_iA_{k-h}e_j\longrightarrow\text{Hom}_{A-Gr}(Ae_i[h],Ae_j[k])$,
given by $\rho (x)(a)=ax$, for all $a\in Ae_i$, is an isomorphism of
$K$-vector spaces, so that $f=\rho_x$, for a unique $x\in
e_iA_{k-h}e_j$. Similarly, there is a unique $y\in e_jA_{h-k}e_i$
such that $f^{-1}=\rho_y$. We then get that $yx=e_i$ and $xy=e_j$.
If $i\neq j$, this is  a contradiction since
$e_iAe_j+e_jAe_i\subseteq J^{gr}(A)$. Therefore we necessarily have
$i=j$ and,  in case $A$ is basic,  we also have $h=k$ for otherwise
we would have that $yx=e_i\in J^{gr}(A)$, which is absurd.

 On the other hand, the map
$\rho:e_iA_0e_i\longrightarrow\text{End}_{A-Gr}(Ae_i[h])$ given
above is an isomorphism of algebras. Therefore each $Ae_i[h]$ has a
local endomorphism algebra in $A-Gr$. Since each finitely generated
graded left $A$-module is an epimorphic image of a finite direct sum
of modules of the form $Ae_i[h]$, we conclude that the category
$A-grproj$ of finitely generated projective graded left $A$-module
is a Krull-Schmidt one, with any indecomposable object isomorphic to
some $Ae_i[h]$. This proves assertion 2 and 3.

As in the ungraded case, the fact that $\text{End}_{A-Gr}(Ae_i[h])$
is a local algebra implies that $J^{gr}(A)e_i[h]$ is the unique
maximal graded submodule of $Ae_i[h]$. If $S_i:=Ae_i/J^{gr}(A)e_i$,
then $S_i[h]$ is a graded-simple module, for each $h\in H$, and all
graded-simple left modules are of this form, up to isomorphism.
Since the projection $Ae_i[h]\longrightarrow S_i[h]$ is a projective
cover in $A-Gr$ we conclude that each graded-simple left $A$-module
has a projective cover in $A-Gr$. From this argument we immediately
get assertion 5, while assertion 4 follows as in the ungraded case.

Finally, the definition of weakly basic
graded algebra is left-right symmetric, so that the last statement
of the proposition also follows.
\end{proof}

\section{Graded pseudo-Frobenius algebras}

\subsection{Definition and characterizations}

 We look at $K$ as an $H$-graded algebra such that $K_h=0$, for
$h\neq 0$.  If $V=\oplus_{h\in H}V_h$ is a  graded $K$-vector space,
then  its dual $D(V)$ gets identified with the graded $K$-vector
space $\oplus_{h\in H}\text{Hom}_K(V_h,K)$, with
$D(V)_h=\text{Hom}_K(V_{-h},K)$ for all $h\in H$.

\begin{defi}
Let $V=\oplus_{h\in H}V_h$ and $W=\oplus_{h\in H}W_h$  be graded
$K$-vector spaces, where the homogeneous components are finite
dimensional, and let $d\in H$ be any element. A bilinear form
$(-,-):V\times W\longrightarrow K$ is said to be \emph{of degree
$d$} if $(V_h,W_k)\neq 0$ implies that $h+k=d$. Such a form will be
called \emph{nondegenerate} when the induced maps $W\longrightarrow
D(V)$ ($w\rightsquigarrow (-,w)$)) and $V\longrightarrow D(W)$
($v\rightsquigarrow (v,-)$)) are bijective.
\end{defi}

Note that, in the above situation, if $(-,-):V\times
W\longrightarrow K$ is a nondegenerate bilinear form of degree $d$,
then the bijective map $W\longrightarrow D(V)$ ( resp
$V\longrightarrow D(W)$) given above gives an isomorphism of graded
$K$-vector spaces
 $W[d]\stackrel{\cong}{\longrightarrow}D(V)$ (resp. $V[d]\stackrel{\cong}{\longrightarrow}D(W)$).

 The following concept is fundamental for us.

 \begin{defi} \label{defi.graded Nakayama form}
 Let $A=\oplus_{h\in H}A_h$ be a weakly basic graded algebra with
 enough idempotents. A bilinear form $(-,-):A\times A\longrightarrow
 K$ is said to be a \emph{graded Nakayama form} when the following
 assertions hold:

 \begin{enumerate}
 \item $(ab,c)=(a,bc)$, for all $a,b,c\in A$
 \item For each $i\in I$ there is a unique $\nu (i)\in I$ such that
 $(e_iA,Ae_{\nu (i)})\neq 0$ and the assignment $i\rightsquigarrow\nu
 (i)$ defines a bijection $\nu :I\longrightarrow I$. \item There is
 a map $\mathbf{h}:I\longrightarrow H$ such that the
 induced map $(-,-):e_iAe_j\times e_jAe_{\nu (i)}\longrightarrow K$ is a
 nondegenerated graded bilinear form  degree $h_i=\mathbf{h}(i)$,
 for all $i,j\in I$.
 \end{enumerate}
The bijection $\nu$ is called the \emph{Nakayama permutation} and
$\mathbf{h}$ will be called the \emph{degree map}. When $\mathbf{h}$
is a constant map and $\mathbf{h}(i)=h$, we
 will say that $(-,-):A\times A\longrightarrow K$ is a \emph{graded
 Nakayama form of degree $h$}.
 \end{defi}

 \begin{defi} \label{defi.locally Noetherian}
 A graded algebra with enough idempotents $A=\oplus_{h\in H}A_h$
 will be called \emph{left (resp. right) locally Noetherian} when
 $Ae_i$ (resp. $e_iA$) satisfies ACC on graded submodules, for each $i\in I$. We will
 simply say that it is locally Noetherian when it is left and right
 locally Noetherian.
 \end{defi}

 Recall that a graded module is \emph{finitely cogenerated} when its
 graded socle is finitely generated and essential as a graded submodule. Recall also that a Quillen
 exact category $\mathcal{E}$ (e.g. an abelian category) is said to be a
 \emph{Frobenius category} when it has enough projectives and enough
 injectives and the projective and the injective objects are the
 same in $\mathcal{E}$.

 \begin{teor} \label{teor.graded pseudo-Frobenius algebras}
 Let $A=\oplus_{h\in H}A_h$ be a weakly basic graded algebra with
 enough idempotents. Consider the following assertions:

 \begin{enumerate}
 \item $A-Gr$ and $Gr-A$ are Frobenius categories
 \item $D(_AA)$ and $D(A_A)$ are projective graded  $A$-modules
 \item The projective finitely generated objects and the injective
 finitely cogenerated  objects coincide in $A-Gr$ (resp. $Gr-A$)
 \item There exists a graded Nakayama form $(-,-):A\times A\longrightarrow
 K$.
 \end{enumerate}
 Then  the following chain of implications holds:
\begin{center}
$1)\Longrightarrow 2)\Longrightarrow 3)\Longleftrightarrow 4)$.
\end{center}

  When $A$ is  graded locally
 bounded, also $4)\Longrightarrow 2)$ holds. Finally, if $A$ is
 graded locally Noetherian, then
 the four  assertions are equivalent.
 \end{teor}
 \begin{proof}
$1)\Longrightarrow 2)$ By proposition \ref{prop.the functor D over
algebras with s.i.}, we have a natural isomorphism

\begin{center}
$HOM_A(?,D(A_A))\cong D(A\otimes_A?):A-Gr\longrightarrow K-Gr$,
\end{center}
and the second functor is exact. Then also the first is exact, which
is equivalent to saying that $D(A_A)$ is an injective object of
$A-Gr$ (see \cite{NVO}[Lemma I.2.4]). A symmetric argument proves
that $D(_AA)$ is injective in $Gr-A$. Then both $D(A_A)$ and
$D({}_{A}A)$ are projective in $A-Gr$ and $Gr-A$ since these are
Frobenius categories.

 $2)\Longrightarrow 3)$ The duality
 $D:A-lfdgr\stackrel{\cong^{op}}{\longleftrightarrow}lfdgr-A:D$
 exchanges projective and injective objects, and, also,   simple
 objects on the left and on the right. Since $A$ is locally finite dimensional all finitely generated left
or right
 graded $A$-modules are locally finite dimensional. Moreover, our
 hypotheses guarantee that each finitely generated projective graded
 $A$-module $P$ is the projective cover of a finite direct sum of simple
 graded modules. Then $D(P)$ is the injective envelope in $A-lfdgr$ of a finite
 direct sum of simple objects. We claim that each injective object $E$ of
 $A-lfdgr$ is an injective object of $A-Gr$. Indeed if $U$ is a graded left
 ideal of $A$, $h\in H$ is any element and $f:U[h]\longrightarrow E$ is morphism in $A-Gr$,
 then we want to prove that $f$ extends to $A[h]$. By an appropriate use of Zorn lemma, we can assume without
 loss of generality that there is no graded submodule  $V$ of $A[h]$
 such that $U[h]\subsetneq V$ and $f$ is extendable to $V$. The
 task is  then reduced to prove that $U=A$. Suppose this is not the
 case, so that there exist $i\in I$ and a homogeneous element $x\in
 Ae_i$ such that $x\not\in U$. But then $Ax+U/U$ is a locally
 finite dimensional graded $A$-module since so is $Ax$. It follows
 that $\text{Ext}_{A-lfdgr}(\frac{U+Ax}{U}[h],E)=0$, which implies that
 $f:U[h]\longrightarrow E$ can be extended to $(U+Ax)[h]$, thus giving a
 contradiction. Now the obvious graded version of Baer's criterion (see \cite{NVO}[Lema
I.2.4]) holds and $E$ is injective in $A-Gr$. In our situation, we
conclude that $D(P)$ is a finitely cogenerated injective object of
$A-Gr$, for each finitely generated projective object $P$ of $Gr-A$.

  Conversely, if $S$ is a simple graded right
 $A$-modules and $p:P\longrightarrow D(S)$ is a projective cover,
 then $D(p):S\cong DD(S)\longrightarrow D(P)$ is an injective
 envelope. This proves that the injective envelope in $A-Gr$ of any
 simple object, and hence any finitely cogenerated injective object of $A-Gr$,  is locally finite dimensional.

 Let  now  $E$ be any locally finite dimensional graded left
 $A$-module. We then get that $E$ is an injective finitely
 cogenerated object of $A-Gr$
 if, and only if, $E\cong D(P)$ for some finitely generated
 projective graded right $A$-module $P$. This implies that $E$ is
 isomorphic to a finite direct sum of graded modules of the form $D(e_iA[-h_i])\cong
 D(e_iA)[h_i]$, where  $h_i\in H$. We then assume,
 without loss of generality, that $E=D(e_iA)[h]$, for some $i\in I$
 and  $h\in H$. Since $e_iA[-h]$ is a direct summand of $A[-h]$ in
 $Gr-A$, assertion 2 implies that $E$ is a projective object in
 $A-Gr$.  Then $E$ is isomorphic to a direct summand of a direct sum of graded modules of the form $Ae_i[h_i]$.
 From the fact that $E$ has a finitely
 generated essential graded socle we easily
 derive that $E$ is a direct summand of $\oplus_{1\leq k\leq
 r}Ae_{i_k}[h_{i_k}]$, for some indices $i_k\in I$. Therefore each finitely cogenerated injective object of $A-Gr$ is  finitely
 generated projective. The analogous fact is true for graded right
 $A$-modules.

 On the other hand, if $P$ is a finitely generated projective graded
 left $A$-module, then $D(P)$ is a finitely cogenerated injective
 objet of $Gr-A$ and, by the previous paragraph, we know that $D(P)$
 is finitely generated projective. We then get that $P\cong DD(P)$
 is finitely cogenerated in $A-Gr$.

$3)\Longrightarrow 4)$ From assertion 3) we obtain its left-right
symmetric statement by applying the duality
$D:A-lfdgr\stackrel{\cong^{op}}{\longleftrightarrow}lfdgr-A:D$,
bearing in mind that an injective object in $lfdgr-A$ is also
injective in $Gr-A$. It  follows that $D(e_iA)$ is an indecomposable
finitely generated projective left $A$-module, for each $i\in I$. We
then get a unique index $\nu (i)\in I$ and an $h_i\in H$ such
that $D(e_iA)\cong Ae_{\nu (i)}[h_i]$. We then have a map $\nu
:I\longrightarrow I$. By the same reason, given another $j\in I$, we have
that $D(Ae_j)\cong e_{\mu (j)}A[k_j]$, for a unique  $\mu (j)\in I$
and $k_j\in H$. We then get

\begin{center}
 $e_iA\cong DD(e_iA)\cong D(Ae_{\nu
(i)}[h_i])\cong D(Ae_{\nu (i)})[-h_i]\cong e_{\mu\nu (i)}A[k_{\nu
(i)}-h_i]$,
\end{center}
and, by proposition \ref{prop.basic graded algebra}, we conclude
that $\mu\nu (i)=i$, for all $i\in I$. This and its symmetric
argument prove that the maps $\mu$ and $\nu$ are mutually inverse.

We fix an isomorphism of graded left $A$-modules $f_i:Ae_{\nu
(i)}[h_i]\stackrel{\cong}{\longrightarrow}D(e_iA)$, for each $i\in
I$. Then we get a bilinear map

\begin{center}
$e_iA\times Ae_{\nu (i)}\stackrel{1\times f_i}{\longrightarrow}
e_iA\times D(e_iA)\stackrel{\text{can}}{\longrightarrow}K$.
\end{center}
Note that we have
$(a,cb)=f_i(cb)(a)=[cf_i(b)](a)=f_i(b)(ac)=(ac,b)$, for all
$(a,b)\in e_iA\times Ae_{\nu (i)}$ and all $c\in A$. This bilinear
form is clearly nondegenerate because $e_iA$ is locally finite
dimensional and, due to the duality $D$, the canonical bilinear form
$e_iA\times D(e_iA)\longrightarrow K$ is nondegenerate, and actually
graded of degree $0$ since
$D(e_iA)_k=D(e_iA_{-k})=\text{Hom}_K(e_iA_{-k},K)$, for each $k\in
H$. On the other hand, if $s,t\in H$ and $a\in e_iA_s$ and $b\in
A_te_{\nu (i)}$ are homogeneous elements, then the degree of $b$ in
$Ae_{\nu (i)}[h_i]$ is $t-h_i$. We get that $(a,b)\neq 0$ if, and
only if, $s+(t-h_i)=0$. This shows that the given bilinear form is
graded of degree $h_i$.

We then define an obvious bilinear form $(-,-):A\times
A\longrightarrow K$ such that $(e_iA,Ae_j)=0$, whenever $j\neq\nu
(i)$, and whose restriction to $e_iA\times Ae_{\nu (i)}$ is the
graded bilinear form of degree $h_i$ given above, for each $i\in I$.
Since $(a,b)=\sum_{i,j\in I}(e_ia,be_j)=\sum_{i\in I}(e_ia,be_{\nu
(i)})$, we get that $(ac,b)=(a,cb)$, for all $a,b,c\in A$, and,
hence, that $(-,-):A\times A\longrightarrow K$ is a graded Nakayama
form.

$4)\Longrightarrow 3)$ Let $(-,-):A\times A\longrightarrow K$ be a
graded Nakayama form and let $\nu :I\longrightarrow I$ and
$\mathbf{h}:I\longrightarrow H$ be the maps given in definition
\ref{defi.graded Nakayama form}. We put $\mathbf{h}(i)=h_i$, for
each $i\in I$. Since the restriction of $(-,-):e_iA\times Ae_{\nu
(i)}\longrightarrow K$ is a nondegenerate graded bilinear form of
degree $h_i$, we get induced isomorphisms of graded $K$-vector
spaces $f_i:Ae_{\nu (i)}[h_i]\longrightarrow D(e_iA)$ and
$g_i:e_{\nu^{-1}(i)}A[h_i]\longrightarrow D(Ae_i)$, where
$f_i(b)=(-,b):x\rightsquigarrow (x,b)$ and
$g_i(a)=(a,-):y\rightsquigarrow (a,y)$. The fact that
$(ac,b)=(a,cb)$, for all $a,b,c\in A$ implies that $f_i$ is a
morphism in $A-Gr$ and $g_i$ is a morphism in $Gr-A$. Therefore the
projective finitely generated objects and the injective finitely
cogenerated objects coincide in $A-lfdgr$ and $lfdgr-A$. By our
comments about the graded Baer criterion, assertion 3 follows
immediately.

$3),4)\Longrightarrow 2)$ We assume  that $A$ is graded locally
bounded. The hypotheses imply that the injective
 finitely cogenerated objects of $A-Gr$ and $Gr-A$ are locally finite
 dimensional and they coincide with the finitely generated
 projective modules. But in this case  $A$ is  locally finite
 dimensional both as a left and as a right graded $A$-module.
 Indeed,  given $i\in I$, one has $e_iA_h=\oplus_{j\in I}e_iA_he_j$.
 By the graded locally bounded condition of $A$ almost all summands
 of this direct sum are zero. This gives that, for each $(i,h)\in I\times
 H$, the vector spaces $e_iA_h$ is finite dimensional, whence,
 that
 $_AA$ is in $A-lfdgr$. Similarly, we get that $A_A\in A-lfdgr$. It
 follows that $D(_AA)$ and $D(A_A)$ are locally finite dimensional.

We claim that $D(A_A)$ is isomorphic to $\oplus_{i\in I}D(e_iA)$
 which, together with assertion 3, will give that $D(A_A)$ is a
 projective graded left $A$-module. This plus its symmetric argument will then
 finish the proof. Indeed $\oplus_{i\in I}D(e_iA)$ is identified with the set of $f\in D(A_A)$ such that $f(e_iA)=0$, for almost all $i\in I$. But we have $D(A_A)=D({}_AA_A)$ (see Remark \ref{rem.dual of bimodules and modules} and Definition \ref{defi.locally f.d. graded algebra}). This means that, for each $g\in D(A_A)$, we have that $g(e_iA_he_j)=0$ for almost all $(i,j,h)\in I\times I\times H$. We then get that $g\in\oplus_{i\in I}D(e_iA)$, so that $D(A_A)=\oplus_{i\in I}D(e_iA)$.

$3),4)\Longrightarrow 1)$ We assume  that $A$ is graded locally
Noetherian. Then $A-Gr$ and $Gr-A$ are locally Noetherian
Grothendieck categories, i.e., the Noetherian objects form a set (up to isomorphism) and
generate both categories. Then each injective object in $A-Gr$ or
$Gr-A$ is a direct sum of indecomposable injective objects and each
direct sum of injective objects is again injective (see
\cite{G0}[Proposition IV.6 and Theorem IV.2]). Since, by hypothesis,
$Ae_i$ and $e_iA$ are injective objects in $A-Gr$ and $Gr-A$,
respectively, we deduce that each projective object in any of these
categories is injective.

On the other hand, $Ae_i$ (resp. $e_iA$) is a Noetherian object of
$A-Gr$ (resp. $Gr-A$), which implies by duality that $D(Ae_i)$
(resp. $D(e_iA)$) is an artinian object of $lfdgr-A$ (resp.
$A-lfdgr$), and hence also of $Gr-A$ (resp. $A-Gr$). But we have
$D(Ae_i)\cong e_{\nu^{-1}(i)}A[h_{\nu^{-1}(i)}]$ (resp.
$D(e_iA)\cong Ae_{\nu (i)}[h_i]$), where $\nu$ is the Nakayama
permutation. By the bijectivity of $\nu$, we get that all $Ae_j$ and
$e_jA$ are Artinian (and Noetherian) objects, whence they have
finite length. Therefore $A-Gr$ and $Gr-A$ have a set of generators
of finite length, which easily implies that the graded socle of each
object in these categories is a graded essential submodule. In
particular, each injective object in $A-Gr$ (resp. $Gr-A$) is the
injective envelope of its graded socle. But if $\{S_t:$ $t\in T\}$
is a family of simple objects of $A-Gr$ (resp. $Gr-A$) and
$\iota_t:S_t\rightarrowtail E(S_t)$ is an injective envelope in
$A-Gr$ (resp. $Gr-A$), then the induced map $\iota:=\oplus_{t\in
T}:\oplus_{t\in T}S_t\rightarrowtail\oplus_{t\in T}E(S_t)$ is an
injective envelope in $A-Gr$ (resp. $Gr-A$) since the direct sum of
injectives is injective. Since each $E(S_t)$ is finitely
cogenerated, whence projective by hypothesis, it follows that each
injective object in $A-Gr$ (resp. $Gr-A$) is projective.
\end{proof}

 \begin{defi} \label{defi.pseudo-Frobenius graded algebras}
 A weakly basic  locally finite dimensional
 graded algebra satisfying any of the
  conditions 3 and 4 of the previous proposition will be said to be
  a \emph{graded pseudo-Frobenius algebra}. When $A$ satisfies
  condition 1, it will be called \emph{graded Quasi-Frobenius}.
 \end{defi}

By definition, if $(-,-):A\times A\longrightarrow K$ is a graded Nakayama form for $A$, then its associated Nakayama permutation $\nu$ and degree map $\mathbf{h}$ are uniquely determined by $(-,-)$. However, one can naturally ask if they just depend on $A$. The following result gives an answer:

\begin{prop} \label{prop.uniqueness of Nakayama permutation}
Let $A$ be a graded pseudo-Frobenius algebra, with $(e_i)_{i\in I}$ as a weakly basic distinguished family of orthogonal idempotents. The following assertions hold:

\begin{enumerate}
\item $\text{Soc}_{gr}({}_AA)=\text{Soc}_{gr}(A_A)$. We will denote this ideal by $\text{Soc}_{gr}(A)$.
\item All graded Nakayama forms for $A$ have the same Nakayama permutation. It assigns to each $i\in I$ the unique $\nu (i)\in I$ such that $e_i\text{Soc}_{gr}(A)e_{\nu (i)}\neq 0$.
\item All graded Nakayama forms for $A$ do not have the same degree map.
\end{enumerate}
\end{prop}

\begin{proof}
1), 2)  Let $(-,-):A\times A\longrightarrow K$ be a graded Nakayama form
for $A$. We have seen in the proof of the implication
$4)\Longrightarrow 3)$ in theorem \ref{teor.graded pseudo-Frobenius
algebras} that then $D(e_iA)\cong Ae_{\nu (i)}[h_i]$. Due to
Proposition \ref{prop.basic graded algebra}, we get that $\nu (i)$ is
independent of $(-,-)$. Moreover, by duality, we get an isomorphism
$e_iA\cong D(Ae_{\nu (i)})[-h_i]$, which induces an isomorphism
between the graded socles. But the graded socle of $D(Ae_j)$ is
isomorphic to $S_j:=e_jA/e_jJ^{gr}(A)$, for each $j\in I$. We then
get that
$\text{Soc}^{gr}(e_iA)e_{j}[-h]\cong\text{Hom}_{A-Gr}(S_j[h],e_i\text{Soc}^{gr}(A))=0$,
for all $j\neq\nu (i)$ and $h\in H$.  Therefore $\nu (i)$ is the unique element of $I$ such that $\text{Soc}^{gr}(e_iA))e_{\nu (i)}\neq 0$.

Put $S=\text{Soc}_{gr}(e_iA)$ for simplicity. The two-sided graded ideal $AS$ is, as a graded right module, a sum of copies of graded-simple right $A$-modules of the form $\frac{e_{\nu (i)}A}{e_{\nu (i)}J}[h]$, where $J=J^{gr}(A)$ and $h\in I$. In particular, we have $AS\subseteq\text{Soc}_{gr}(A_A)=\oplus_{j\in I}\text{Soc}_{gr}(e_jA)$. If we had $AS\not\subseteq S$, we  would have that $\text{Soc}_{gr}(e_jA)\cap AS\neq 0$, for some $j\neq i$. But we know that $\text{Soc}_{gr}(e_jA)\cong\frac{e_{\nu (j)}A}{e_{\nu (j)}J}[-h_j]$. It would follow that we have an isomorphism $\frac{e_{\nu (j)}A}{e_{\nu (j)}J}[-h_j]\cong \frac{e_{\nu (i)}A}{e_{\nu (i)}J}[h]$, for some $h\in H$. Taking projective covers in $\text{Gr}-A$, we then get that $e_{\nu (j)}A[-h_j]\cong e_{\nu (i)}A[h]$, for some $h\in H$. By Proposition \ref{prop.basic graded algebra}, we would get that $\nu (j)=\nu (i)$, which would give that $j=i$ and, hence,  a contradiction. Therefore we have $AS=S$ and, hence, $S$ is a two-sided ideal.

To end the proof of assertions 1 and 2, we use the fact that $\text{Soc}_{gr}({}_AA)$ is essential as a graded left ideal of $A$, because each $e_iA$ is finitely cogenerated injective (see Theorem \ref{teor.graded pseudo-Frobenius algebras} and Definition \ref{defi.pseudo-Frobenius graded algebras}). We then have a homogeneous  element $0\neq w\in S\cap\text{Soc}_{gr}({}_AA)$ such that $Aw$ is a graded-simple left $A$-module. Bearing in mind that $Se_{\nu (i)}=S$, we get that $Aw=\text{Soc}_{gr}(Ae_{\nu (i)})$. Therefore we have that $\text{Soc}_{gr}(Ae_{\nu (i)})\subseteq\text{Soc}_{gr}(e_iA)$. The bijective condition of $\nu $ then gives that $\text{Soc}_{gr}({}_AA)\subseteq\text{Soc}_{gr}(A_A)$. The opposite inclusion follows by symmetry.

3) In the example \ref{exem.pseudo-Frobenius}(4) below, the map $[-,-]:A\times A\longrightarrow K$ given by $[f,g]=(f,xg)$ is a graded Nakayama form of degree $m-1$ for $A$.
\end{proof}

The following result complements theorem \ref{teor.graded pseudo-Frobenius
 algebras} and gives a handy criterion, in the locally Noetherian
 case, for $A$ to be graded Quasi-Frobenius.

 \begin{cor} \label{cor.Frobenius=simple socle}
 Let $A=\oplus_{h\in H}A_h$ be a weakly basic graded algebra with  $(e_i)_{i\in I}$ as a weakly basic distinguished family of orthogonal idempotents. The following
 assertions are equivalent:

 \begin{enumerate}

\item $A$ is locally Noetherian and the following conditions hold:

\begin{enumerate}
 \item For each $i\in I$, $Ae_i$ and $e_iA$ have a simple essential
socle in $A-Gr$ and $Gr-A$, respectively \item There are bijective
maps $\nu ,\nu' :I\longrightarrow I$ such that
$\text{Soc}_{gr}(e_iA)\cong \frac{e_{\nu (i)}A}{e_{\nu
 (i)}J^{gr}(A)}[h_i]$ and  $\text{Soc}_{gr}(Ae_i)\cong
 \frac{Ae_{\nu '
 (i)}}{[J^{gr}(A)e_{\nu '
 (i)}}[h'_i]$, for certain $h_i,h'_i\in H$

\end{enumerate}

\item $A$ is a locally Noetherian pseudo-Frobenius graded algebra.

\item $A$ is graded
Quasi-Frobenius
 \end{enumerate}
 \end{cor}
 \begin{proof}
$2)\Longrightarrow 3)$ is given by theorem \ref{teor.graded pseudo-Frobenius
 algebras}.

$3)\Longrightarrow 2)$ By hypothesis, in $A-Gr$  every countable direct sum of injective envelopes of simple objects is an injective object. Then one proves that each $Ae_i$ is a Noetherian object of $A-Gr$ as in the ungraded unital case (see the proof of \cite{Kasch}[Theorem 6.5.1]). By symmetry, one also gets that each $e_iA$ is a Noetherian object of $Gr-A$.

$2)\Longrightarrow 1)$ By  duality and the fact that the finitely generated projective and finitely cogenerated injective graded modules coincide, we know that $\text{Soc}_{gr}(e_iA)$ is a simple object of $lfdgr-A$ which is an essential graded submodule of $e_iA$. If $\nu$ is the Nakayama permutation, then, by proposition \ref{prop.uniqueness of Nakayama permutation},  we have that
$\text{Soc}_{gr}(e_iA)e_{\nu (i)}\neq 0$. This implies that $\text{Hom}_{Gr-A}(e_{\nu (i)}A[h],\text{Soc}_{gr}(e_iA))\neq 0$, for some $h\in H$. It follows that $\text{Soc}_{gr}(e_iA)\cong\frac{e_{\nu (i)}A}{e_{\nu (i)}J^{gr}(A)}[h]$.
A left-right symmetric argument proves that the map $\nu'=\nu^{-1}$ satisfies the condition in the statement.

 $1)\Longrightarrow 2)$  By definition of weakly  basic,  $A$ is locally
 finite dimensional, so that $Ae_i$ and $e_iA$ are locally finite
 dimensional modules, for all $i\in I$. It then follows by duality that $D(e_iA)$ is an Artinian object of $A-Gr$, for all $i\in
 I$. By the same reason, we get that $D(e_iA)$  has a
 unique simple essential quotient in $A-Gr$, meaning that $D(e_iA)$ has
 a unique maximal superfluous subobject in this category. By a classical argument, it
 follows that $D(e_iA)$  has a
 projective cover in $A-Gr$, which is an epimorphism
 of the form $p:Ae_j[h]\twoheadrightarrow D(e_iA)$. It  follows from
 this
 that $D(e_iA)$ is a Noetherian object, whence, an object of finite
 length in $A-Gr$ since it is a quotient of a Noetherian object.
 With this and its symmetric argument we get that all finitely
 cogenerated injective objects in $A-Gr$ and $Gr-A$ have finite
 length, which implies by duality that also the finitely generated
 projective objects have finite length.

 The fact that $\text{Soc}_{gr}(e_iA)$ is graded-simple implies that
 the injective envelope of $e_iA$ in $A-Gr$ is of the form $\iota :e_iA\rightarrowtail E\cong
 D(Ae_j)[h]$, while the projective cover of $E$ is of the form $p:e_kA[h']\twoheadrightarrow
 E$. Then $\iota$ factors through $p$ yielding a monomorphism $u:e_iA\rightarrowtail
 e_kA[h']$. But then the graded socles of $e_iA$ and $e_kA[h']$ are
 isomorphic. By condition 1.b) and the weakly basic condition, this implies that $i=k$.
 By comparison of graded composition lengths, we get that $u$ is an
 isomorphism, which in turn implies that both $p$ and $\iota$ are
 also isomorphisms. Therefore all the $e_iA$, and hence all finitely
 generated projective objects, are finitely cogenerated injective
 objects of $A-Gr$. The left-right symmetry of assertion 1 implies
 that the analogous fact is true in $Gr-A$. Then, applying duality,
 we get that the finitely generated projective objects and the
 finitely cogenerated injective objects coincide in $A-Gr$ and
 $Gr-A$.
\end{proof}

\subsection{Some examples}

 \begin{ejems} \label{exem.pseudo-Frobenius}
 The following are examples of graded pseudo-Frobenius algebras over a field
 $K$:

 \begin{enumerate}
 \item When $H=0$ and $A=\Lambda$ is a finite-dimensional
 self-injective algebra, which is equivalent to saying that
 $\Lambda$ is quasi-Frobenius.

\item The group algebra $KH$, with $(KH)_h=Kh$ for each $h\in H$, is  graded simple both as a left and as a right graded module over itself. It immediatetly follows that $KH$ is
graded Quasi-Frobenius algebra.

\item When $\Lambda$ is any finite-dimensional split basic algebra and $A=\hat{\Lambda}$ is its repetitive algebra,
  in the
 terminology of \cite{H}, then $A$ is a (trivially graded) quasi-Frobenius algebra with enough idempotents (see
 op.cit.[Chapter II]).

 \item Take the $\mathbb{Z}$-graded algebra $A=K[x,x^{-1},y,z]/(y^2,z^2)$,
 where
 $\text{deg}(x)=\text{deg}(y)=\text{deg}(z)=1$. Given any integer
 $m$, we have a canonical basis
 $\mathcal{B}_m=\{x^m,x^{m-1}y,x^{m-1}z,x^{m-2}yz\}$ of $A_m$. Consider the
 graded bilinear form $A\times A\longrightarrow K$ of degree $m$
 identified  by the fact that if $f\in A_n$ and $g\in A_{m-n}$, then
 $(f,g)$ is the coefficient of $x^{m-2}yz$ in the expression of $fg$
 as a $K$-linear combination of the elements of $\mathcal{B}_m$.
 Then $(-,-)$ is a graded Nakayama form for $A$, so that $A$ is
 graded pseudo-Frobenius.
 \end{enumerate}
 \end{ejems}

\subsection{The Nakayama automorphism and  construction of  Nakayama forms}

We remind the reader that if $\sigma$ and $\tau$ are graded automorphisms of $A$ (of zero degree), then we can form a graded $A-A-$bimodule ${}_\sigma A_\tau$ as follows. We take ${}_\sigma A_\tau =A$ as a graded $K$-vector space and we define $a\cdot x\cdot b=\sigma (a)x\tau (b)$, for all $a,x,b\in A$.

 \begin{cor} \label{cor.Nakayama automorphism}
 Let $A=\oplus_{h\in H}A_h$ be a  graded pseudo-Frobenius algebra, let $(e_i)_{i\in I}$ be
 a weakly basic
 distinguished family of orthogonal idempotents and let $D(A)=D({}_AA_A)$ the dual graded bimodule of ${}_AA_A$. The
 following assertions hold:

 \begin{enumerate}
  \item There is an automorphism
 of (ungraded) algebras
 $\eta:A\longrightarrow A$, which permutes the idempotents $e_i$ and maps homogeneous elements of $\bigcup_{i,j\in I}e_iAe_j$ onto homogeneous elements, such
 that $_1A_\eta$ is isomorphic to $D(A)$ as an ungraded
 $A$-bimodule. \item If the map $\mathbf{h}:I\longrightarrow H$ associated to the  Nakayama form $(-,-):A\times A\longrightarrow
 K$ takes constant value $h$, then $\eta$ can be chosen to be graded and such that
 $D(A)$ is isomorphic to $_1A_\eta [h]$ as graded $A$-bimodules.
 \end{enumerate}
 \end{cor}
 \begin{proof}
 By definition, $D(A)$ consists of the linear forms $f:A\longrightarrow K$ such that $f(e_iA_he_j)=0$, for all but finitely many triples $(i,j,h)\in I\times I\times H$. On the other hand, the left $A$-module $\oplus_{i\in I}D(e_iA)$ can be identified with the space of linear forms $g:A=\oplus_{i\in I}e_iA\longrightarrow K$ such that $g(e_iA)=0$, for almost all $i\in I$, and fixed any $i\in J$, the set of pairs $(j,h)\in I\times H$ such that $g(e_iA_he_j)\neq 0$ is finite. It follows that, when viewed as $K$-subspaces of $\text{Hom}_K(A,K)$, we have $D(A)=\oplus_{i\in I}D(e_iA)$ and, by a left-right symmetric argument, we also have $D(A)=\oplus_{i\in I}D(Ae_i)$.

 1) Let us fix a graded Nakayama form $(-,-):A\times A\longrightarrow
 K$ and associated maps $\nu :I\longrightarrow I$ and $\mathbf{h} :I\longrightarrow
 H$. The assignment $b\rightsquigarrow (-,b)$ gives an isomorphism
 of graded  $A$-modules
 $Ae_{\nu (i)}[h_i]\stackrel{\cong}{\longrightarrow}D(e_iA)$, where $h_i=\mathbf{h}(i)$, for
 each $i\in I$. By taking the direct sum of all these maps, we get
 an isomorphism of ungraded left $A$-modules $_AA\longrightarrow\oplus_{i\in
 I}D(e_iA)=D(A)$.
 Therefore the assignment  $b\rightsquigarrow (-,b)$  actually gives
 an isomorphism ${}_AA\stackrel{\cong}{\longrightarrow}{}_AD(A)$.
 Symmetrically, the assignment $a\rightsquigarrow (a,-)$ gives an
 isomorphism $A_A\stackrel{\cong}{\longrightarrow}D(A)_A$. It then
 follows that, given $a\in A$, there is a unique $\eta (a)\in A$
 such that $(a,-)=(-,\eta (a))$. This gives a $K$-linear map $\eta :A\longrightarrow
 A$ which, by its own definition, is bijective. Moreover, given $a,b,x\in
 A$, we get

 \begin{center}
 $(x,\eta (ab))=(ab,x)=(a,bx)=(bx,\eta (a))=(b,x\eta (a))=(x\eta (a),\eta (b)=(x,\eta (a)\eta
 (b))$,
 \end{center}
 which shows that $\eta (ab)=\eta (a)\eta (b)$, for all $a,b\in A$.
 Therefore $\eta$ is an automorphism of $A$ as an ungraded algebra.
 Moreover if $0\neq a\in e_iAe_j$, then $(a,-)$ vanishes on all
 $e_{i'}Ae_{j'}$ except in $e_jAe_{\nu (i)}$. Therefore
 $(-,\eta (a))$ does the same. By definition of the Nakayama form,
 we necessarily have $\eta (a)\in e_{\nu (i)}Ae_{\nu (j)}$. We claim
 that if $a\in e_iAe_j$ is an element of degree $h$, then  $\eta (a)$ is an element of degree $h+h_j-h_i$. Indeed,
 let $h'\in H$ be such that $\eta (a)_{h'}\neq 0$. Then  $(-,\eta(a)_{h'}):e_jAe_{\nu (i)}\longrightarrow
 K$ is a nonzero linear form which vanishes on $e_jA_ke_{\nu (i)}$, for all
 $k\neq h_j-h'$. Let us pick up $x\in e_jA_{h_j-h'}e_{\nu (i)}$ such that $(x,\eta (a)_{h'})\neq
 0$. Then we have that $(x,\eta (a))=(x,\eta (a)_{h'})\neq 0$, due
 to the fact that $(-,-):e_jAe_{\nu (i)}\times e_{\nu (i)}Ae_{\nu (j)}\longrightarrow
 K$ is a graded bilinear form of degree $h_j$. We then get that $0\neq (x,\eta
 (a))=(a,x)$, which implies that $h+(h_j-h')=h_i$, and hence that $h'=h+(h_j-h_i)$. Then $h'$ is uniquely determined by
 $a$, so that $\eta (a)$ is homogeneous of degree $h+h_j-h_i$ as
 desired.

 Putting $a=e_i$ in the previous paragraph, we get that $\eta (e_i)\in e_{\nu (i)}Ae_{\nu (i)}$ has degree $0$, and
 then $\eta (e_i)$ is an idempotent element of the local algebra
 $e_{\nu (i)}A_0e_{\nu (i)}$. It follows that $\eta (e_i)=e_{\nu (i)}$, for each $i\in I$.

 Finally, we consider the $K$-linear isomorphism $f:A\longrightarrow
 D(A)$ which maps $b\rightsquigarrow (-,b)=(\eta^{-1}(b),-)$. We
 readily see that $f$ is a homomorphism of left $A$-modules.
 Moreover, we have equalities

 \begin{center}
 $(a,b\eta (b'))=(ab,\eta (b'))=(b',ab)=(b'a,b)=[f(b)b'](a)$,
 \end{center}
 which shows that $f$ is a homomorphism of right $A$-modules $A_\eta\longrightarrow$

 $D(A)$. Then $f$ is an isomorphism
 $_1A_\eta\stackrel{\cong}{\longrightarrow}D(A)$.

 2) The proof of assertion 1 shows that if $\mathbf{h}(i)=h$, for all $i\in
 I$, then $\eta$ is a graded automorphism of degree $0$. Moreover,
 the isomorphism $f:_1A_\eta\stackrel{\cong}{\longrightarrow}D(A)$ is the direct sum of the isomorphisms
 of graded left $A$-modules $f_i:Ae_{\nu
 (i)}[h]\stackrel{\cong}{\longrightarrow}D(e_iA)$ which map $b\rightsquigarrow (-,b)$. It
 then follows that $f$ is an isomorphism of graded bimodules
 $_1A_\eta [h]\stackrel{\cong}{\longrightarrow}D(A)$.
\end{proof}

To finish this subsection, we will see that if one knows that $A$ is
split graded pseudo-Frobenius, then all possible graded Nakayama
forms for $A$ come in similar way. Recall that if $V=\oplus_{h\in
H}V_h$ is a graded vector space, then its \emph{support}, denoted
$\text{Supp}(V)$,  is the set of $h\in H$ such that $V_h\neq 0$.

\begin{prop} \label{prop.graded Nak-form via basis}
Let $A$ be a split pseudo-Frobenius graded algebra and $(e_i)_{i\in
I}$ a weakly basic distinguished family of orthogonal idempotents.
The following assertions hold:

\begin{enumerate}
\item If $h_i\in \text{Supp}(e_i\text{Soc}_{gr}(A))$, then
$\text{dim}(e_i\text{Soc}_{gr}(A))_{h_i})=1$ \item For a bilinear
form $(-,-):A\times A\longrightarrow K$,  the following statements
are equivalent:

\begin{enumerate}
\item $(-,-)$ is a graded Nakayama form for $A$ \item There exist
an element $\mathbf{h}=(h_i)\in\prod_{i\in
I}\text{Supp}(e_i\text{Soc}_{gr}(A))$ and a basis $\mathcal{B}_i$ of
$e_iA_{h_i}e_{\nu (i)}$, for each $i\in I$, such that:

\begin{enumerate}
\item $\mathcal{B}_i$ contains a (unique) element $w_i$ of
$e_i\text{Soc}_{gr}(A)_{h_i}$ \item If $a,b\in\bigcup_{i,j}e_iAe_j$
are homogeneous elements, then $(e_iA_h,A_ke_j)=0$ unless $j=\nu
(i)$ and $h+k=h_i$ \item If $(a,b)\in e_iA_h\times A_{h_i-h}e_{\nu
(i)}$, then $(a,b)$ is the coefficient of $w_i$  in the expression
of $ab$ as a linear combination of the elements of $\mathcal{B}_i$.
\end{enumerate}
\end{enumerate}
\end{enumerate}
\end{prop}

\begin{proof}

1) Let us fix $h_i\in\text{Supp}(e_i\text{Soc}^{gr}(A)))$ and
suppose that $\{x,y\}$ is a linearly independent subset of
$e_i\text{Soc}^{gr}(A))_{h_i}$. We then have $xA=yA$ since
$e_i\text{Soc}^{gr}(A))$ is graded-simple. We get from this that
also  $xA_0=yA_0$. By proposition \ref{prop.basic graded algebra},
we know that  $J(A_0)=J^{gr}(A)_0$   and the split hypothesis on $A$
implies that $A_0=J(A_0)\oplus (\oplus_{j\in I}Ke_j)$. It follows
that $Kx=x(\oplus_{j\in I}Ke_j)=xA_0=yA_0=y(\oplus_{j\in
I}Ke_j)=Ky$, which contradicts the linear independence of $\{x,y\}$.

2) $b)\Longrightarrow a)$ By Proposition \ref{prop.uniqueness of Nakayama permutation}, the Nakayama permutation
is completely determined by $A$. The given element $\mathbf{h}$ is
then interpreted as a map $I\longrightarrow H$, which will be our
degree function. Moreover, we claim that the induced graded bilinear form $e_iAe_j\times e_jAe_{\nu (i)}\longrightarrow K$ is nondegenerate. Indeed, let $0\neq a\in e_iA_he_j$ be a homogeneous element.  Then $aA\cap e_i\text{Soc}_{gr}(A)\neq 0$ since $e_i\text{Soc}_{gr}(A)=\text{Soc}_{gr}(e_iA)$ is essential in $e_iA$. Taking $b\in e_jA_{h_i-h}e_{\nu (i)}$ homogeneous such that $ab\in e_i\text{Soc}_{gr}(A)$, from conditions b.i and b.iii we derive that $(a,b)\neq 0$.

It only remains to check that $(ab,c)=(a,bc)$, for
all $a,b,c$. This easily reduces to the case when $a,b,c$ are
homogeneous and there are indices $i,j,k$ such that $a=e_iae_j$,
$b=e_jbe_k$ and $c=e_kce_{\nu (i)}$. But in this case, we have
$(a,bc)=(ab,c)=0$ when
$\text{deg}(a)+\text{deg}(b)+\text{deg}(c)\neq h_i$. On the other
hand, by condition b.iii), if
$\text{deg}(a)+\text{deg}(b)+\text{deg}(c)=h_i$ then $(ab,c)$ and
$(a,bc)$ are both the coefficient of $w_i$ in the expression of
$abc$ as linear combination of the elements of $\mathcal{B}_i$. So
the equality $(ab,c)=(a,bc)$ holds, for all $a,b,c\in A$.

$a)\Longrightarrow b)$ We first take a basis $\mathcal{B}^0$ of
$A_0$ such that $\mathcal{B}^0=\{e_i:$ $i\in I\}\cup
(\mathcal{B}^0\cap J(A_0))$ and
$\mathcal{B}^0\subseteq\bigcup_{i,j\in I}e_iA_0e_j$. The graded
Nakayama form gives by restriction a nondegenerate bilinear map

\begin{center}
$(-,-):e_iA_0e_i\times e_iA_{h_i}e_{\nu (i)}\longrightarrow K$.
\end{center}
We choose as $\mathcal{B}_i$ the basis of $e_iA_{h_i}e_{\nu (i)}$
which is right orthogonal to $e_i\mathcal{B}^0e_i$ with respect to
this form. As usual, if $b\in e_i\mathcal{B}^0e_i$, we denote by
$b^*$ the element of $\mathcal{B}_i$ such that
$(c,b^*)=\delta_{bc}$, where $\delta_{bc}$ is the Kronecker symbol.
We then claim that $w_i:=e_i^*$ is in $e_i\text{Soc}_{gr}(A)$. This
will imply that $h_i\in\text{Supp}(\text{Soc}_{gr}(A))$ and, due to
assertion 1, we will get also that $w_i$ is the only element of
$e_i\text{Soc}_{gr}(A)_{h_i}$ in $\mathcal{B}_i$. Indeed suppose
that $w_i\not\in e_i\text{Soc}_{gr}(A)$. We then have $a\in
J^{gr}(A)$ such that $aw_i\neq 0$. Without loss of generality, we
assume that $a$ is homogeneous and that $a=e_jae_i$, for some $j\in
I$. Then $0\neq aw_i\in e_jAe_{\nu (i)}$, which implies the
existence of a homogeneous element $b\in e_iAe_j$ such that
$(b,aw_i)\neq 0$ since the induced graded bilinear form
$e_iAe_j\times e_jAe_{\nu (i)}\longrightarrow K$ is nondegenerate.
But then we have $(ba,w_i)\neq 0$ and $\text{deg}(ba)=0$ since the
induced graded bilinear form $e_iAe_i\times e_iAe_{\nu
(i)}\longrightarrow K$ is of degree $h_i$. But  $ba\in
e_iJ^{gr}(A)_0e_i=e_iJ(A_0)e_i$ and, by the choice of the basis
$\mathcal{B}^0$, each element of $e_iJ(A_0)e_i$ is a linear
combination of the elements in $\mathcal{B}^0\cap e_iJ(A_0)e_i$. By
the  choice of $w_i$, we have $(c,w_i)=0$, for all $c\in
\mathcal{B}^0\cap e_iJ(A_0)e_i$. It then follows that $(ba,w_i)=0$,
which is a contradiction.

It is now clear that  conditions b.i and b.ii hold. In order to
prove b.iii, take $(a,b)\in e_iA_h\times A_{h_i-h}e_{\nu (i)}$. We
then have $(a,b)=(e_i,ab)$, where $ab\in e_iA_{h_i}e_{\nu (i)}$. Put
$ab=\sum_{c\in\mathcal{B}_i}\lambda_cc$, where $\lambda_c\in K$ for
each $c\in \mathcal{B}_i$. We then get
$(a,b)=(e_i,\sum_{c}\lambda_cc)=\sum_{c}\lambda_c(e_i,c)=\lambda_{w_i}$,
i.e., $(a,b)$ is the coefficient of $w_i$ in the expression
$ab=\sum_{c}\lambda_cc$.
\end{proof}

\begin{defi} \label{defi.graded Nakayama form associated to basis}
Let $A=\oplus_{h\in H}A_h$ be a split pseudo-Frobenius graded
algebra, with $(e_i)_{i\in I}$ as weakly basic distinguished family
of idempotents and $\nu :I\longrightarrow I$ as Nakayama
permutation. Given a pair $(\mathcal{B},\mathbf{h})$ consisting of
an element $\mathbf{h}=(h_i)_{i\in I}$ of $\prod_{i\in
I}\text{Supp}(e_i\text{Soc}_{gr}(A))$ and a family
$\mathcal{B}=(\mathcal{B}_i)_{i\in I}$, where $\mathcal{B}_i$ is a
basis of $e_iA_{h_i}e_{\nu (i)}$ containing an element of
$e_iSoc_{gr}(A)$,  for each $i\in I$, we call \emph{graded Nakayama
form associated to $(\mathcal{B},\mathbf{h})$} to the bilinear form
$(-,-):A\times A\longrightarrow K$ determined by the conditions b.ii
and b.iii of last proposition. When $\mathbf{h}$ is constant, i.e.
there is $h\in H$ such that $h_i=h$ for all $i\in I$,  we will call
$(-,-)$ the graded Nakayama form of $A$ of degree $h$ associated to
$\mathcal{B}$.
\end{defi}

\subsection{Graded algebras given by quivers and relations}
\label{section.Graded quivers and relations}

Recall that a \emph{quiver} or \emph{oriented graph} is a quadruple
$Q=(Q_0,Q_1,i,t)$, where $Q_0$ and $Q_1$ are sets, whose elements
are called \emph{vertices} and \emph{arrows} respectively, and
$i,t:Q_1\longrightarrow Q_0$ are maps. If $a\in Q_1$ then $i(a)$ and
$t(a)$ are called the \emph{origin} (or $\emph{initial vertex}$) and
the \emph{terminus} of $a$.

Given a quiver $Q$, a \text{path} in $Q$ is a concatenation of
arrows $p=a_1a_2...a_r$ such that $t(a_k)=i(a_{k+1})$, for all
$k=1,...,r$. In such case, we put $i(p)=i(a_1)$ and $t(p)=t(a_r)$
and call them the origin and terminus of $p$. The number $r$ is the
\emph{length} of $p$ and we view the vertices of $Q$ as paths of
length $0$. The \emph{path algebra} of $Q$, denoted by $KQ$,  is the
$K$-vector space with basis the set of paths, where the
multiplication extends by $K$-linearity the multiplication of paths.
This multiplication is defined as $pq=0$, when $t(p)\neq i(q)$, and
$pq$ is the obvious concatenation path, when $t(p)=i(q)$. The
algebra $KQ$ is an algebra with enough idempotents, where $Q_0$ is
a distinguished family of orthogonal idempotents. If $i\in Q_0$ is a
vertex, we will write it as $e_i$ when we view it as an element of
$KQ$.

\begin{defi} \label{defi.graded quiver}
Let $H$ be an abelian group. An \emph{(H-)graded quiver} is a pair
$(Q,\text{deg})$, where $Q$ is a quiver and
$\text{deg}:Q_1\longrightarrow H$ is a map, called the \emph{degree
or weight function}. $(Q,\text{deg})$ will be called \emph{locally
finite dimensional} when, for each $(i,j,h)\in Q_0\times Q_0\times
H$, the set of arrows $a$ such that
$(i(a),t(a),\text{deg}(a))=(i,j,h)$ is finite.
\end{defi}

We will simply say that $Q$ is an $H$-graded quiver, without mention
to the degree function which is implicitly understood. Each degree
function on a quiver $Q$ induces an $H$-grading on the algebra $KQ$,
where the degree of a path of positive length is defined as the sum
of the degrees of its arrows and $\text{deg}(e_i)=0$, for all $i\in
Q_0$. In the following result, for each natural number $n$, we
denote by $KQ_{\geq n}$ the vector subspace of $KQ$ generated by the
paths of length $\geq n$. For each ideal $I$ of an algebra, we put
$I^\omega =\bigcap_{n>0}I^n$.

\begin{prop} \label{prop.algebra by quiver and relations}
Let $A=\oplus_{h\in H}A_h$ be a split basic locally finite
dimensional graded algebra with enough idempotents and let
$J=J^{gr}(A)$ be its graded Jacobson radical. There is an $H$-graded
locally finite dimensional  quiver $Q$ and a subset
$\rho\subset\bigcup_{i,j\in Q_0} e_iKQ_{\geq 2}e_j$, consisting of
homogeneous elements with respect to the induced $H$-grading on
$KQ$, such that $A/J^{\omega}$ is isomorphic to $KQ/<\rho >$.
Moreover $Q$ is unique, up to isomorphism of $H$-graded quivers.
\end{prop}
\begin{proof}
It is an adaptation of the corresponding proof, in more restrictive
situations, of the ungraded case (see, e.g., \cite{BG}[Section 2]).
We give the general outline, leaving the details to the reader.

Let $(e_i)_{i\in I}$ be the basic distinguished family of orthogonal
idempotents. The graded quiver $Q$ will have $Q_0=I$ as its sets of
vertices. Whenever $h\in\text{Supp}(\frac{e_iJe_j}{e_iJ^2e_j})$, we
will select a subset $Q_1(i,j)_h$ of $e_iJ_he_j$ whose image by the
projection
$e_iJ_he_j\twoheadrightarrow\frac{e_iJ_he_j}{e_i(J^2)_he_j}$ gives
a basis of $\frac{e_iJ_he_j}{e_i(J^2)_he_j}$. We will take as
arrows of degree $h$ from $i$ to $j$ the elements of $Q_1(i,j)_h$,
and then $Q_1=\bigcup_{i,j\in Q_0;h\in H}Q_1(i,j)_h$. The
so-obtained graded quiver gives a grading on $KQ$ and there is an
obvious homomorphism of graded algebras $f:KQ\longrightarrow A$
which takes $e_i\rightsquigarrow e_i$ and $a\rightsquigarrow a$, for
all $i\in Q_0$ and $a\in Q_1$.

We claim that the composition
$KQ\stackrel{f}{\longrightarrow}A\stackrel{p}{\twoheadrightarrow}A/J^\omega$
is surjective or, equivalently, that $\text{Im}(f)+J^\omega =A$. Due
to the split basic condition of $A$, it is easy to see that
$A=(\sum_{i\in I}Ke_i)\oplus J$ and the task is then reduced to
prove the inclusion $J\subseteq \text{Im}(p)+J^\omega$. Since
$e_iA_he_j$ is finite-dimensional, for each triple $(i,h,j)\in
I\times H\times I$, there is a smallest natural number $m_{ij}(h)$
such that $e_i(J^n)_he_j=e_i(J^{n+1})_he_j$, for all $n\geq
m_{ij}(h)$. We will prove, by induction on $k\geq 0$, that
$e_i(J^{m_{ij}(h)-k})_he_j\subset \text{Im}(f)+J^\omega$, for all
$(i,h,j)$, and then the inclusion $J\subseteq \text{Im}(p)+J^\omega$
will follow. The case $k=0$ is trivial, by the definition of
$m_{ij}(h)$. So we assume that $k>0$ in the sequel. Fix any triple
$(i,h,j)\in I\times H\times I$ and put $n:=m_{ij}(h)-k$. If $x\in
e_i(J^n)_he_j$ then $x$ is a sum of products of the form
$x_1x_2\cdot ...\cdot x_n$, where $x_r$ is a homogeneous element in
$e_{i'}Je_{j'}$, for some pair $(i',j')\in I\times I$. So it is not
restrictive to assume that $x=x_1x_2\cdot ...\cdot x_n$ is a product
as indicated. By definition of the arrows of $Q$, each $x_r$ admits
a decomposition $x_r=y_r+z_r$, where $y_r$ is a linear combination
of arrows (of the same degree) and $z_r\in J^2$. It follows that
$x=y+z$, where $y$ is a linear combination of paths of length $n$
and $z\in e_iJ^{n+1}e_j$. Then $y\in \text{Im}(f)$ and, by the
induction hypothesis, we know that $z\in \text{Im}(f)+J^\omega$.

Proving that $\text{Ker}(p\circ f)\subseteq KQ_{\geq 2}$ goes as in
the ungraded case,  as so does the proof of the uniqueness of $Q$.
Both are left to the reader.
\end{proof}

A weakly basic locally finite dimensional algebra $A$  will be
called \emph{connected} when, for each pair $(i,j)\in I\times I$
there is a sequence $i=i_0, i_1,...,i_n=j$ of elements of $I$ such
that, for each $k=1,...,n$, either $e_{i_{k-1}}Ae_{i_k}\neq 0$ or
$e_{i_k}Ae_{i_{k-1}}\neq 0$. If $Q$ is a  graded quiver, we say that
$Q$ is a  \emph{connected graded quiver} when it is connected as an
ungraded quiver. That is, when, for each pair $(i,j)\in Q_0\times
Q_0$, there is a sequence $i=i_0, i_1,...,i_n=j$ of vertices such
that there is an arrow $i_{k-1}\rightarrow i_k$ or an arrow
$i_k\rightarrow i_{k-1}$, for each $k=1,...,n$.

\begin{cor} \label{cor.positively graded quiver algebras}
Let $A=\oplus_{n\geq 0}A_n$ be a split basic locally finite
dimensional positively $\mathbb{Z}$-graded such that $A_0$ is semisimple. Then there exists a $\mathbb{Z}$-graded quiver $Q$ with $deg(a)>0$ for all $a\in Q_1$, uniquely determined up to
isomorphism of $\mathbb{Z}$-graded quivers, such that $A$ is
isomorphic to $KQ/I$, for a homogeneous ideal $I$ of $KQ$ such that
$I\subseteq KQ_{\geq 2}$. If, moreover, $A$ is connected   and the equality
$A_n=A_1\cdot\stackrel{n}{...}\cdot A_1$ holds for all $n>0$, then
the following assertions are equivalent:

\begin{enumerate}
\item $A$ is graded pseudo-Frobenius  \item There exists a graded
Nakayama form $(-,-):A\times
A\longrightarrow K$ with constant degree function.
\end{enumerate}
In particular, the Nakayama automorphism $\eta$ is always graded in
this latter case.
\end{cor}
\begin{proof}
The point here is that if $x\in J^n$ is a homogeneous element, then
$\text{deg}(x)\geq n$, which implies that $J^\omega$ does not
contain nonzero homogeneous elements and, hence, that $J^\omega =0$. Then
the first part of the statement is a direct consequence of
proposition \ref{prop.algebra by quiver and relations}. Moreover,
one easily sees that the connectedness of $A$ is equivalent in this
case to the connectedness of the quiver $Q$.

As for the second part, we only need to prove that if $(-,-):A\times
A\longrightarrow K$ is a graded Nakayama form, then its associated
degree function is constant. The argument is inspired by
\cite{MV}[Proposition 3.2]. We consider that $A=KQ/I$, where $Q$ is
connected. The facts that $A_0$ is semisimple and
$A_n=A_1\cdot\stackrel{n}{...}\cdot A_1$, for all $n>0$,
translate into the fact that the degree function
$\text{deg}:Q_1\longrightarrow\mathbb{Z}$ takes constant value $1$,
so that the induced grading on $KQ$ is the one by path length.

Let now $\eta :A\longrightarrow A$ be the Nakayama automorphism
associated to $(-,-)$. If $a:i\rightarrow j$ is any arrow in $Q$,
then  from  corollary \ref{cor.Nakayama automorphism} we get that
$\eta (a)$ is a homogeneous element in $e_{\nu (i)}Je_{\nu
(j)}=e_{\nu (i)}Ae_{\nu (j)}$.  Since obviously $\text{deg}(\eta (a))\neq
0$, we get that $\text{deg}(\eta (a))\geq\text{deg}(a)$, which
implies that $\text{deg}(\eta (x))\geq\text{deg}(x)$, for each
homogeneous element $x\in\bigcup_{i,j\in Q_0}e_iAe_j$. Let again $a:i\rightarrow j$ be an
arrow and put $x=\eta^{-1}(a)$. We claim that $x$ is homogeneous of
degree $1$. Indeed, we have $x=x_1+x_2+...+x_n$, with
$\text{deg}(x_k)=k$, so that $a=\eta (x)=\eta (x_1)+\eta (x')$,
where $x'=\sum_{2\leq k\leq n}x_k$ and, hence,   $\eta (x')$ is a
sum of homogeneous elements of degrees $\geq 2$. It follows that
$a=\eta (x_1)$ and $\eta (x')=0$, which, by the bijective condition
of $\eta$, gives that $x'=0$. Therefore $x=x_1$ as desired.

The last paragraph implies that, for each pair $(i,j)\in Q_0\times
Q_0$ such that there is an arrow $i\rightarrow j$ in $Q$, there is a
vector subspace $V_{ij}$ of $e_{\nu^{-1}(i)}KQ_1e_{\nu^{-1}(j)}$
such that $\eta_{|V_{ij}}: V_{ij}\longrightarrow e_iKQ_1e_j$ is a
bijection. Let now $\tilde{Q}$ be the subquiver of $Q$ with the same
vertices and with arrows those $a\in Q_1$ such that $\text{deg}(\eta
(a))=1$. Then $V_{ij}\subseteq
e_{\nu^{-1}(i)}K\widetilde{Q}e_{\nu^{-1}(j)}$ and
$\tilde{A}=\frac{K\tilde{Q}+I}{I}$ is a subalgebra of $A=KQ/I$ such
that the image of the restriction map
$\eta_{|\tilde{A}}:\tilde{A}\longrightarrow A$ contains the vertices
and the arrows (when viewed as elements of $A$ in the obvious way).
Note that $\eta_{|\tilde{A}}$ is a homomorphism of graded algebras,
which immediately implies that it is surjective and, hence,
bijective. But then necessarily $\tilde{A}=A$ for $\eta$ is an
injective map. We will derive from this that $\text{deg}(\eta
(a))=\text{deg}(a)$, for each $a\in Q_1$. Indeed, if
$\text{deg}(\eta (a))>1$,  then $\eta (a)=\eta (x)$, for some
homogeneous element $x\in\tilde{A}$ of degree
$\text{deg}(x)=\text{deg}(\eta (a))$. By the injective condition of
$\eta$, we would get that $a=x$, which is a contradiction.

  If now $h:Q_0\longrightarrow\mathbb{Z}$ is the degree
function associated to the graded Nakayama form, the proof of
corollary \ref{cor.Nakayama automorphism} gives that
$h_{i(a)}=h_{t(a)}$, for each $a\in Q_1$. Due to the connectedness
of $Q$, we conclude that $h$ is a constant function.
\end{proof}

\section{Coverings  of graded categories and preservation of the pseudo-Frobenius condition}

\subsection{Covering theory of graded categories}

In this part we will present the basics of covering theory of graded
categories or, equivalently, of graded algebras with enough
idempotents. It is an adaptation of the classical theory (see
\cite{Ri},  \cite{G2}, \cite{BG}), where we incorporate  more recent
ideas of \cite{CM} and \cite{Asa2}, where some of the constraining
hypotheses of the initial theory disappear.

Let $A=\oplus_{h\in H}A_h$ and $B=\oplus_{h\in H}B_h$ be two locally
finite dimensional graded algebras with enough idempotents, with
$(e_i)_{i\in I}$ and $(\epsilon_j)_{j\in J}$ as respective
distinguished families of homogeneous orthogonal idempotents of
degree $0$. Suppose that $F:A\longrightarrow B$ is a graded functor
and that it is surjective on objects, i.e., for each $j\in J$ there
exists $i\in I$ such that $F(e_i)=\epsilon_j$. To this functor one
canonically associates the  \emph{pullup or restriction of scalars
functor} $F^\rho:B-Gr\longrightarrow A-Gr$. If $X$ is a graded left
$B$-module, then we put $e_iF^\rho (X)=\epsilon_{F(i)}X$, for all
$i\in I$, and if $a\in\bigcup_{i,i'\in I}e_iAe_{i'}$ and $x\in
F^\rho (X)$, then $ax:=F(a)x$. It has a left adjoint $F_\lambda
:A-Gr\longrightarrow B-Gr$, called the \emph{pushdown functor},
whose precise definition will be given below in the case that we
will need in this work.

The procedure of taking a weak skeleton gives rise to a graded
functor as above. Indeed, suppose that $A$ is as above and consider
the equivalence relation $\sim$ in $I$ such that $i\sim i'$ if, and
only if,  $Ae_i$ and $Ae_{i'}$ are isomorphic graded $A$-modules. If
$I_0$ is a set of representatives under this relation, then we can
consider the full graded subcategory of $A$ having as objects the
elements of $I_0$. This amounts to take the graded subalgebra
$B=\oplus_{i,i'\in I_0}e_iAe_{i'}$, which will be called the
\emph{weak skeleton} of $A$. If we denote by $[i]$ the unique
element of $I_0$ such that $i\sim [i]$, then there are  elements
$\xi_i\in e_iA_0e_{[i]}$ and $\xi_i^{-1}\in e_{[i]}A_0e_i$ such that
$\xi_i\xi_i^{-1}=e_i$ and $\xi_i^{-1}\xi_i=e_{[i]}$. We fix $\xi_i$
and $\xi_i^{-1}$ from now on.  By convention, we assume that
$\xi_{[i]}=e_{[i]}$, for each $[i]\in I_0$.  Now we get a surjective
on objects graded functor $F:A\longrightarrow B$ which takes
$i\rightsquigarrow [i]$ on objects and if $a\in e_iAe_{i'}$, then
$F(a)=\xi_i^{-1}a\xi_{i'}$. If we take $P=\oplus_{i\in I_0}e_iA$
then $P$ is an $H$-graded $B-A-$bimodule and the pullup functor is
naturally isomorphic to the 'unitarization' of the graded Hom
functor, $A\text{HOM}_B(P,-):B-Gr\longrightarrow A-Gr$ (see
subsection \ref{subsect. gr-algebras enough-idempt}). It is an
equivalence of categories and the pushout functor $F_\lambda$ gets
identified with $P\otimes_A-:A-Gr\longrightarrow B-Gr$, which, up to
isomorphism, takes $M\rightsquigarrow\oplus_{i\in I_0}e_iM$.

\begin{defi} \label{defi.Covering functor}
 Let $A$ and $B$ be as above. A graded functor $F:A\longrightarrow B$ will be called a {\bf
covering functor} when it is surjective on objects and, for each
$(i,j,h)\in I\times J\times H$, the induced maps

\begin{center}
$\oplus_{i'\in F^{-1}(j)}e_iA_he_{i'}\longrightarrow e_{F(i)}B_he_j$

$\oplus_{i'\in F^{-1}(j)}e_{i'}A_he_{i}\longrightarrow
e_{j}B_he_{F(i)}$
\end{center}
are bijective.
\end{defi}

We shall now present the paradigmatic example of covering functor,
which is actually the only one that we will need in our work.  In
the rest of this subsection, $A=\oplus_{h\in H}A_h$ will be a
locally finite dimensional graded algebra with a distinguished
family $(e_i)_{i\in I}$ of homogeneous orthogonal idempotents of
degree $0$, fixed from now on. We will assume that $G$ is a group
acting on $A$ as a group of graded automorphisms (of degree $0$)
which permutes the $e_i$. That is, if $\text{Aut}^{gr}(A)$ denotes
the group of graded automorphisms of degree $0$ which permute the
$e_i$, then we have a group homomorphism $\varphi
:G\longrightarrow\text{Aut}^{gr}(A)$. We will write $a^g=\varphi
(g)(a)$, for each $a\in A$ and $g\in G$. In such a case, the
\emph{skew group algebra} $A\star G$ has as elements the formal
$A$-linear combinations $\sum_{g\in G}a_g\star g$, with $a_g\in A$
for all $g\in G$.

The multiplication extends by linearity the product $(a\star
g)(b\star g')=ab^g\star gg'$, where $a,b\in A$ and $g,g'\in G$. The
new algebra inherits an $H$-grading from $A$ by taking $(A\star
G)_h=A_h\star G=\{\sum_{g\in G}a_g\star g\in A\star G:$ $a_g\in
A_h\text{, for all }g\in G\}$. We have a canonical inclusion of
$H$-graded algebras $\iota: A\hookrightarrow A\star G$ which maps
$a\rightsquigarrow a\star 1$, where $1$ is the unit of $G$.

\begin{prop} \label{prop.G-covering functor}
In the situation above, let $\Lambda$ be the weak skeleton of
$A\star G$ and $F:A\star G\longrightarrow\Lambda$ the corresponding
functor. The following assertions hold:

\begin{enumerate}
\item $(A\star G)(e_i\star 1)$ and $(A\star G)(e_j\star 1)$ (resp. $(e_i\star 1)(A\star G)$ and $(e_j\star 1)(A\star G)$) are isomorphic in $A\star G-\text{Gr}$ (resp. $\text{Gr}-A\star G$) if, and only if, $i$ and $j$ are in the same $G$-orbit.

\item The composition $A\stackrel{\iota}{\hookrightarrow}
A\star G\stackrel{F}{\longrightarrow}\Lambda$ is a covering functor.
The corresponding pushdown functor $F_\lambda
:A-Gr\longrightarrow\Lambda-Gr$ is exact and takes
$Ae_i\rightsquigarrow\Lambda e_{[i]}$, for each $i\in I$.
\end{enumerate}
\end{prop}
\begin{proof}
All through the proof, whenever necessary,  we identify $a$ with $a\star 1$, for each $a\in A$.

1) We do the part of the assertion concerning graded left modules, leaving the one for right modules to the reader. We have that $(A\star
G)e_i\cong (A\star G)e_j$ if, and only if, there are $x\in
e_i(A\star G)_0e_j=\oplus_{g\in G}e_iA_0e_{g(j)}\star g$ and  $y\in
e_j(A\star G)_0e_i=\oplus_{g\in G}e_jA_0e_{g(i)}\star g$ that
$xy=e_i$ and $yx=e_j$. This immediately implies that $i$ and $j$ are
in the same $G$-orbit, i.e., that $e_i^g=e_j$, for some $g\in G$.
Conversely,   we have equalities $e_i\star
g=(e_i\star 1)(e_i\star g)(e_{g^{-1}(i)}\star 1)$ and
$e_{g^{-1}(i)}\star g^{-1}=(e_{g^{-1}(i)}\star 1)(e_{g^{-1}(i)}\star
g^{-1})(e_i\star 1)$, and also $(e_i\star g)(e_{g^{-1}(i)}\star
g^{-1})=e_i\star 1$ and $(e_{g^{-1}(i)}\star g^{-1})(e_i\star
g)=e_{g^{-1}(i)}\star 1$, which show that $(A\star G)e_i\cong
(A\star G)e_{g^{-1}(i)}$ for all $g\in G$ and $i\in I$.

\vspace*{0.3cm}

2)  The pullup functor is the composition $\Lambda
-Gr\stackrel{F^\rho}{\longrightarrow}A\star
G-Gr\stackrel{\iota^\rho}{\longrightarrow} A-Gr$, so that the
pushdown functor is $F_\lambda\circ\iota_\lambda$. We know that
$F_\lambda$ is an equivalence of categories. On the other hand
$\iota_\lambda$ is naturally isomorphic to $A\star
G\otimes_A-:A-Gr\longrightarrow A\star G-Gr$ since $\iota^\rho$ is
the usual restriction of scalars. The exactness of $A\star
G\otimes_A-$ implies that of $F_\lambda\circ\iota_\lambda$ and the
action of this functor on projective objects takes
$Ae_i\rightsquigarrow (A\star G)\otimes_AAe_i\cong (A\star
G)e_i\rightsquigarrow F_\lambda ((A\star G)e_i)$. But this latter
graded $\Lambda$-module is isomorphic to $\Lambda e_{[i]}$ by the
explicit definition of the pushdown functor when taking a weak
skeleton.

In order to check that $F\circ\iota$ is a covering functor we look
at the definition of the weak skeleton, taking into account assertion 1. We then take exactly one index $i\in I$ in each
$G$-orbit and in that way we get a subset $I_0$ of $I$. Up to graded
isomorphism, we have $\Lambda =\oplus_{i,j\in I_0}e_i(A\star G)e_j$.
For the explicit definition of $F$, we put $\xi_{g(i)}=e_{g(i)}\star
g$ and $\xi_{g(i)}^{-1}=e_i\star g^{-1}$, for each $i\in I_0$ and
$g\in G$. If $g,g'\in G$ and $i,j\in I_0$, then the map
$F:e_{g(i)}(A\star G)e_{g'(j)}\longrightarrow e_i\Lambda
e_j=e_i(A\star G)e_j$ takes
$x\rightsquigarrow\xi_{g(i)}^{-1}x\xi_{g'(j)}=(e_i\star
g^{-1})x(e_{g'(j)}\star g')$. Then the composition

\begin{center}
$e_{g(i)}Ae_{g'(j)}\stackrel{\iota}{\longrightarrow}e_{g(i)}(A\star
G)e_{g'(j)}\stackrel{F}{\longrightarrow} e_i\Lambda e_j=e_i(A\star
G)e_j$
\end{center}
takes $a\rightsquigarrow (e_i\star g^{-1})(a\star 1)(e_{g'(j)}\star
g')=a^{g^{-1}}\star g^{-1}g'$.

The proof that $F\circ\iota$ is a covering functor gets then reduced
to check that if $i,j\in I_0$ and $h\in H$ then the maps

\begin{center}
$\oplus_{g\in G}e_{g(i)}A_he_j\longrightarrow  e_i\Lambda_h
e_j=e_i(A\star G)_he_j$, \hspace*{0.5cm} $(a_g)_{g\in
G}\rightsquigarrow \sum_{g\in G}a_g^{g^{-1}}\star g^{-1}$

$\oplus_{g\in G}e_{i}A_he_{g(j)}\longrightarrow  e_i\Lambda_h
e_j=e_i(A\star G)_he_j$, \hspace*{0.5cm} $(b_g)_{g\in
G}\rightsquigarrow \sum_{g\in G}b_g\star g$
\end{center}
are both bijective. It is clear since $\oplus_{g\in
G}(e_{g(i)}A_he_j)^{g^{-1}}\star g^{-1}=e_i(A\star
G)_he_j=\oplus_{g\in G}e_iA_he_{g(j)}\star g$.
\end{proof}

\begin{defi} \label{defi.G-covering functor}
If $A=\oplus_{h\in H}A_h$, $G$ and $\Lambda$ are as above, then the
functor $F\circ\iota :A\longrightarrow\Lambda$ will be called a
\emph{G-covering} of $\Lambda$.
\end{defi}

If $A$ and $G$ are as in the setting, we say that $G$ \emph{acts
freely on objects} when $g(i)\neq i$, for all $i\in I$ and $g\in
G\setminus\{1\}$. In such case we can form the \emph{orbit category}
$A/G$. The objects of this category are the $G$-orbits $[i]$ of
indices $i\in I$ and the morphisms from $[i]$ to $[j]$ are formal
sums $\sum_{g\in G}[a_g]$, where $[a_g]$ is the $G$-orbit of an
element $a_g\in e_iAe_{g(j)}$. This definition does not depend on
$i,j$, but just on the orbits $[i],[j]$. The anticomposition of
morphisms extends by $K$-linearity the following rule. If
$a,b\in\bigcup_{i,j\in I}e_iAe_j$ and $[a]$, $[b]$ denote the
$G$-orbits of $a$ and $b$, then $[a]\cdot [b]=0$, in case
$[t(a)]\neq [i(b)]$, and $[a]\cdot [b]=[ab^g]$, in case
$[t(a)]=[i(b)]$,  where $g$ is the unique element of $G$ such that
$g(i(b))=t(a)$. We have an obvious canonical projection $\pi
:A\longrightarrow A/G$ with takes $a\rightsquigarrow [a]$. The
following is the classical interpretation of $\Lambda$ and is
implicit in \cite{Asa2}.

\begin{cor} \label{cor.orbit category}
Let $A$,  $G$ and $\Lambda$ be as in proposition
\ref{prop.G-covering functor} and suppose that $G$ acts freely on
objects. There is an equivalence of categories $\Upsilon
:\Lambda\stackrel{\cong}{\longrightarrow} A/G$ such that
$\Upsilon\circ F\circ\iota :A\longrightarrow A/G$ is the canonical
projection.
\end{cor}
\begin{proof}
Let fix a set $I_0$ of representatives of the elements of $I$ under
the equivalence relation $\sim$ given by:  $i\sim j$ if, and only
if, $(A\star G)e_i\cong (A\star G)e_j$  are isomorphic as graded
$(A\star G)$-modules. Then, by definition, $\Lambda$ is the category
having as objects the elements of $I_0$ and $e_i\Lambda
e_j=e_i(A\star G)e_j=\oplus_{g\in G}[e_iAe_{g(j)}\star g]$ as space
of morphisms from $i$ to $j$. The functor $\Upsilon
:\Lambda\longrightarrow A/G$ is defined as $\Upsilon (i)=[i]$, for
each $i\in I_0$, and by $\Upsilon (a\star g)=[a]$, when $g\in G$ and
$a\in e_iAe_{g(j)}$, with $i,j\in I_0$.

The functor is clearly dense. On the other hand, if  $\Upsilon
(\sum_{g\in G}a_g\star g)=\Upsilon (\sum_{g\in G}b_g\star g)$, with
$a_g,b_g\in e_iAe_{g(j)}$ for some $i,j\in I_0$, then we have an
equality of formal finite sums of orbits $\sum_{g\in
G}[a_g]=\sum_{g\in G}[b_g]$. This implies that $[a_g]=[b_g]$, for
each $g\in G$, because if there is an element $\sigma\in G$ such
that $\sigma (a_g)$ and $b_h$ have the same origin an terminus, for
some $h\in G$,  then $\sigma =id$ due to the free action on objects.
But the equality $[a_g]=[b_g]$ also implies that $a_g=b_g$ since
$i(a_g)=i(b_g)=i$. Therefore $\Upsilon$ is a faithful functor.
Finally, the orbit of any homogeneous morphism $a$ in $A$ contains
an element with origin, say $i$,  in $I_0$. Then, in order to prove
that $\Upsilon$ is full,  we can assume that $[a]$ is the orbit of
an element $a\in e_iAe_{g(j)}$, for some $i,j\in I_0$ and some $g\in
G$. But then $a\star g\in e_i(A\star G)e_j$, and we clearly have
$\Upsilon (a\star g)=[a]$.

The equality of functors $\Upsilon\circ F\circ\iota =\pi$ is
straightforward.
\end{proof}

\subsection{Preservation of the pseudo-Frobenius condition}

In this final subsection we will see that the pseudo-Frobenius condition is preserved and reflected by covering functors, under fairly general conditions. Recall that a graded algebra with enough idempotents is \emph{graded semisimple} when as a (left or right) graded module over itself it is a direct sum of graded-simple modules. The following is a useful auxiliary result.

\begin{lema} \label{lem.skew group algebra}
Let $A$ be a weakly basic (H-)graded locally bounded  algebra, where $(e_i)_{i\in I}$ is a weakly basic distinguished family of orthogonal idempotents, let $G$ be a group of graded automorphisms of $A$ which acts freely on objects and let $A\star G$ the skew group algebra with its natural $H$-grading. The following assertions hold:

\begin{enumerate}

\item The graded Jacobson radical of $A\star G$ is $J^{gr}(A)\star G:=\{\sum_{g\in G}a_g\star g\in A\star G:$ $a_g\in J^{gr}(A)\text{, for all }g\in G\}$.
\item $\frac{A\star G}{J^{gr}(A\star G)}$ is isomorphic to the skew group algebra $\frac{A}{J^{gr}(A)}\star G$, which is  graded semisimple.
\item If $S$ is a graded-semisimple left (resp. right) $A$-module, then $(A\star G)\otimes_AS$ (resp. $S\otimes_A(A\star G)$) is a graded-semisimple left (resp. right)  $A\star G$-module.
\end{enumerate}
\end{lema}
\begin{proof}
We will identify $a$ with $a\star 1$, when we want to see the element $a\in A$ as an element of $A\star G$.

\vspace*{0.3cm}

1), 2)  Let us put $J=J^{gr}(A)$ and let $a\in e_iJ_he_j$ be a homogeneous element. We claim that $a\in J^{gr}(A\star G)$, from which we will derive that $J=J\star 1\subset J^{gr}(A\star G)$. It will follow that $J\star G=(J\star 1)(A\star G)\subseteq J^{gr}(A\star G)$ since $J^{gr}(A\star G)$ is a two-sided ideal of $A\star G$. We will prove that  $a(A\star G)_{-h}e_i$ consists of nilpotent elements in the algebra $e_i(A\star G)_0e_i$, which will imply that, for each $y\in e_j(A\star G)_{-h}e_i=\oplus_{g\in G}[e_jA_{-h}e_{g(i)}\star g]$, the element $e_i-ay=e_1\star 1-(a\star 1)y$ is invertible in $e_i(A\star G)_0e_i$, thus settling our claim (see Proposition \ref{prop.graded Jacobson radical}). Let us write $y=y_1\star g_1+...+y_r\star g_r$, where $y_k\in e_jA_{-h}e_{g_k(i)}$ for $k=1,...,r$, so that $ay=\sum_{1\leq k\leq r}ay_k\star g_k$. Note that $ay_k\in e_iA_0e_{g_k(i)}$, for each $k=1,...,r$. Due to the locally bounded condition of $A$ and the free action of $G$ on objects, the set $\{g\in G:$ $e_iA_0e_{g(i)}\neq 0\}$ is finite. Fix now an integer $m>0$ and consider, for each $n>0$, the sequence of products  $\{(ay_{i_1}), (ay_{i_1})(ay_{i_2})^{g_{i_1}},...,(ay_{i_1})(ay_{i_2})^{g_{i_1}}\cdot ...\cdot (ay_{i_n})^{g_{i_1}g_{i_2}...g_{i_{n-1}}}\}$. Bearing in mind that $(ay_{i_1})(ay_{i_2})^{g_{i_1}}\cdot ...\cdot (ay_{i_k})^{g_{i_1}g_{i_2}...g_{i_{k-1}}}\in e_iA_0e_{g_{i_1}g_{i_2}...g_{i_{k-1}}g_{i_k}(i)}$, for each $k>1$, and that $G$ acts freely on objects,  we can always choose $n$ large enough so that, for any choice of $i_1,...,i_n\in\{1,2,...,r\}$,  there is an index $j$ in the $G$-orbit of $i$, such that at least $m+1$ products in the last sequence are in $e_iA_0e_j$. This implies that $(ay_{i_1})(ay_{i_2})^{g_{i_1}}\cdot ...\cdot (ay_{i_n})^{g_{i_1}g_{i_2}...g_{i_{n-1}}}$ is in $e_iA_0e_j\cdot (e_jJ_0e_j)^m$. But, by proposition \ref{prop.basic graded algebra}, we know that $e_jJ_0e_j$ is the Jacobson radical of the finite dimensional algebra $e_jA_0e_j$. Choosing $m$ such that $ (e_jJ_0e_j)^m=0$, we conclude that it is always possible to choose a large enough $n>0$ such that  $(ay_{i_1})(ay_{i_2})^{g_{i_1}}\cdot ...\cdot (ay_{i_n})^{g_{i_1}g_{i_2}...g_{i_{n-1}}}=0$, for any choice of $i_1,...,i_n\in\{1,2,...,r\}$.
 It then follows that $(ay)^n=0$ since $(ay)^n$ is a sum of elements of the form  $(ay_{i_1})(ay_{i_2})^{g_{i_1}}\cdot ...\cdot (ay_{i_n})^{g_{i_1}g_{i_2}}...g_{i_{n-1}}\star g_{i_1}g_{i_1}...g_{i_n} $.

Put now $B=A/J$, which is a graded-semisimple algebra, where $B=\oplus_{i\in I}Be_i$ is a decomposition as a direct sum of graded-simple left $B$-modules (see the proof of proposition \ref{prop.basic graded algebra}). We obviously have that $\frac{A\star G}{J\star G}\cong\frac{A}{J}\star G=B\star G$. If we prove that $B\star G$ is also graded semisimple, then we will have that $0=J^{gr}(B\star G)=\frac{J^{gr}(A\star G)}{J\star G}$ and the proof of assertions 2 and 3 will be finished. Indeed we have a decomposition $B\star G=\oplus_{i\in I}(B\star G)e_i$ and the task reduces to check that $(B\star G)e_i$ is a graded-simple $B\star G$-module, for each $i\in I$. Let $0\neq x\in (B\star G)_he_i$ any homogeneous element. Then $x=\sum_{g\in G}x_g\star g$, where $x_g\in B_he_{g(i)}$ for each $g\in G$. Fix a $\tau\in G$ such that $x_\tau\neq 0$. We have that $0\neq x_\tau^{\tau^{-1}}\in B_he_i$ and the graded-simple condition of $Be_i$ as a graded $B$-module gives an element $y_\tau\in e_iB_{-h}$ such that $y_\tau x_\tau^{\tau^{-1}}=e_i$. We now take $y=y_\tau\star\tau^{-1}$, so that $yx=\sum_{g\in G}y_{\tau}x_g^{\tau^{-1}}\star\tau^{-1}g$. Note that $y_{\tau}x_g^{\tau^{-1}}\in e_iBe_{\tau^{-1} g(i)}$. By the weakly basic condition of $A$, we know that $e_iAe_j\subseteq J$ (and so $e_iBe_j=0$)  whenever $j\neq i$. We then have $y_{\tau}x_g^{\tau^{-1}} =0$, unless $\tau^{-1} g(i)=i$. By the free action of $G$ on objects, we conclude that $y_{\tau}x_g^{\tau^{-1}}\neq 0$ implies that $\tau =g$. We then get that $yx=y_{\tau}x_\tau^{\tau^{-1}}\star 1=e_i\star 1$. It follows that $(B\star G)x=(B\star G)e_i$, as desired.

\vspace*{0.3cm}

3) By assertions 1 and 2, we only need to check that $(A\star G)\otimes_AS$ is annihilated by $J\star G$. But we have an equality $J\star G=(A\star G)(J\star 1)$.  It follows that any tensor $x\otimes z$, with $x\in J\star G$ and $z\in S$, is zero in  $(A\star G)\otimes_AS$ because $JS=0$. Then we get    $(J\star G)[(A\star G)\otimes_AS]=0$.
\end{proof}

\begin{defi}
Let $A=\oplus_{h\in H}A_h$ be a graded pseudo-Frobenius algebra and
$G$ be a group acting on $A$ as graded automorphisms. A graded
Nakayama form $(-,-):A\times A\longrightarrow K$ will be called
\emph{$G$-invariant} when $(a^g,b^g)=(a,b)$, for all $a,b\in A$ and
all $g\in G$.
\end{defi}

The following result shows that, in some circunstances,  a covering functor preserves the pseudo-Frobenius condition.

\begin{prop} \label{prop.gr-Frobenius via pushdown functor}
Let $A=\oplus_{h\in H}A_h$ be a (split basic)  locally
bounded graded pseudo-Frobenius algebra, with $(e_i)_{i\in I}$ as weakly basic distinguished family of
orthogonal homogeneous idempotents,  and let $G$ be a group which
acts on $A$ as graded automorphisms which permute the $e_i$ and
which acts freely on objects. If there exists a $G$-invariant graded Nakayama form
$(-,-):A\times A\longrightarrow K$, then $\Lambda =A/G$ is a (split basic)  locally bounded
graded pseudo-Frobenius algebra whose
graded Nakayama form is induced from $(-,-)$.
\end{prop}
\begin{proof}
We put $\pi:=F\circ\iota$, where $F$ and $\iota$ are as in
proposition \ref{prop.G-covering functor}.  We then know that $\pi$
is surjective on objects and each homogeneous morphism in
$\Lambda$ is a sum of homogeneous morphisms of the form $\pi (a)$,
with $a\in\bigcup_{i,j\in I}e_iAe_j$. We will put $\pi (i)=[i]$ and
$\pi (a)=[a]$, for each $i\in I$ and homogeneous element
$a\in\bigcup_{i,j\in I}e_iAe_j$. Note that $[i]$ and $[a]$ can be
identified with the $G$-orbits of $i$ and $a$ (see corollary
\ref{cor.orbit category}).

We first check that $\Lambda$ is weakly basic.
The functor $F$, which is an equivalence of categories, gives an
isomorphism of algebras $e_i(A\star G)_0e_i\cong e_{[i]}\Lambda_0
e_{[i]}$, for each $i\in I$. But we have $e_i(A\star
G)_0e_i=\oplus_{g\in G}e_iA_0e_{g(i)}\star g$. This algebra is
finite dimensional due to the graded locally bounded condition of
$A$ and the fact that $G$ acts freely on objects. Then all nilpotent
elements of $e_i(A\star G)_0e_i$ belong to its Jacobson radical. It
follows that $\mathbf{m}:=e_iJ(A_0)e_i\oplus (\oplus_{g\neq
1}e_iA_0e_{g(i)}\star g)$ is contained in $J(e_i(A\star G)_0e_i)$
since, due again to the graded locally bounded condition of $A$ and
the free action of $G$, we know that $\mathbf{m}$ consists of
nilpotent elements. Since $\frac{e_i(A\star
G)_0e_i}{\mathbf{m}}\cong\frac{e_iA_0e_i}{e_iJ(A_0)e_i}$ is a
division algebra, we conclude that $\mathbf{m}=J(e_i(A\star
G)_0e_i)$ and that $e_{[i]}\Lambda_0 e_{[i]}\cong e_i(A\star
G)_0e_i$ is a local algebra. Moreover, we have that
$\frac{e_{[i]}\Lambda_0e_{[i]}}{e_{[i]}J(\Lambda_0)e_{[i]}}\cong\frac{e_i(A\star
G)_0e_i}{e_iJ((A\star G)_0)}\cong\frac{e_iA_0e_i}{e_iJ(A_0)e_i}$, so
that $\Lambda$ is split whenever $A$ is so.

We next prove that $e_{[i]}\Lambda_he_{[j]}\subset J^{gr}(\Lambda )$,
whenever $[i]\neq [j]$ or $A$ is basic and $h\neq 0$. This  will give the weakly basic condition of $\Lambda$ and, in case $A$ is basic, also its basic condition. But that amounts to prove that $e_i(A\star
G)_he_j=\oplus_{g\in G}[e_iA_he_{g(j)}\star g]\subset J^{gr}(A\star G)$,  whenever $[i]\neq [j]$ of $A$ is basic and $h\neq 0$, because the pushdown functor
$F_\lambda:A\star G-Gr\longrightarrow\Lambda-Gr$ is an equivalence of
categories. The desired fact follows directly from Lemma \ref{lem.skew group algebra} and the definition of weakly basic and basic algebra.

We pass to define the graded Nakayama form for $\Lambda$. We will
define first graded bilinear forms $<-,->:e_{[i]}\Lambda
e_{[j]}\times e_{[k]}\Lambda e_{[l]}\longrightarrow K$, for all
objects  $[i]$, $[j]$, $[k]$ and $[l]$ of $\Lambda$. When $[j]\neq
[k]$ the bilinear form is zero. In case $[j]=[k]$, we need to define
$<\pi (a),\pi (b)>$ whenever $a\in \oplus_{g,g'\in
G}e_{g(i)}Ae_{g'(j)}$ and $b\in \oplus_{g,g'\in
G}e_{g(j)}Ae_{g'(l)}$. We define $<\pi (a),\pi (b)>$ when $a,b\in
\bigcup_{i,l\in I}e_iAe_l$, with $[t(a)]=[i(b)]=[j]$ and then extend
by $K$-bilinearity to the general case. Indeed we define
$<[a],[b]>=(a,b^g)$, where $g\in G$ satisfies  that $g(i(b))=t(a)$.
Note that $g$ is unique since $G$ acts freely on objects. We leave
to the reader the routine task of checking that
$<-,->:e_{[i]}\Lambda e_{[j]}\times e_{[j]}\Lambda e_{[k]}$ is
well-defined. The graded bilinear form
$<-,->:\Lambda\times\Lambda\longrightarrow K$ is defined as the
'direct sum' of the just defined graded bilinear forms.

We next check that it satisfies all the conditions of definition
\ref{defi.graded Nakayama form}. We first check condition 2 in that
definition. Let $x,y\in\bigcup_{[i],[j]}e_{[i]}\Lambda e_{[j]}$ be
such that $<x,y>\neq 0$. Then we know that there is $j\in I$ such
that $t(x)=[j]=i(y)$. Fix such index $j\in I$. Since the functor
$\pi :A\longrightarrow\Lambda$ is covering it gives bijections
$\oplus_{g\in
G}e_{g(i)}Ae_j\stackrel{\cong}{\longrightarrow}e_{[i]}\Lambda
e_{[j]}$ and $\oplus_{g\in
G}e_{j}Ae_{g(k)}\stackrel{\cong}{\longrightarrow}e_{[j]}\Lambda
e_{[k]}$, for all $G$-orbits of indices $[j]$ and $[k]$. We then put
$x=\sum_{g\in G}\pi (a_g)$ and $y=\sum_{g\in G}\pi (b_g)$ such that
$a_g\in e_{g(i)}Ae_j$ and $b_g\in e_jAe_{g(k)}$, for all $g\in G$.
By definition of $<-,->$, we then have $0\neq <x,y>=\sum_{g,g'\in
G}(a_g,b_{g'})$, which implies that there are $g,g'\in G$ such that
$(a_g,b_{g'})\neq 0$. This implies that $g'(k)=\nu (g(i))$, where
$\nu$ is the Nakayama permutation associated to $(-,-)$. But, due to
the $G$-invariant condition of  $(-,-)$, we have that $\nu
(g(i))=g(\nu (i))$. This shows that $[k]=[\nu (i)]$. It  follows
that $<e_{[i]}\Lambda ,\Lambda e_{[k]}>\neq 0$ implies that
$[k]=[\nu (i)]$. Therefore assertion 2 of definition
\ref{defi.graded Nakayama form} holds, and the bijection
$\bar{\nu}:I/G\stackrel{\cong}{\longrightarrow}I/G$ maps
$[i]\rightsquigarrow [\nu (i)]$.

The $G$-invariance of  $(-,-)$ also implies that if
$\mathbf{h}:I\longrightarrow H$ is the degree function associated to
$(-,-)$,  then $h(g(i))=h(i)$ $\forall i\in I$. As a consequence,
the graded bilinear form $<-,->:e_{[i]}\Lambda\times\Lambda e_{[\nu
(i)]}\longrightarrow K$ is of degree $h_i:=\mathbf{h}(i)$, for each
$i\in I$. Moreover, for each $j\in I$,  the induced graded bilinear form  $<-,->:e_{[i]}\Lambda e_{[j]}\times e_{[j]}\Lambda e_{[\nu
(i)]}\longrightarrow K$  is nondegenerate since so is $(-,-):e_iAe_{g(j)}\times e_{g(j)}Ae_{\nu (i)}\longrightarrow K$, for each $g\in G$.
Then condition 3 of Definition \ref{defi.graded Nakayama form} is satisfied by $<-,->$, by taking  $\bar{\mathbf{h}}:I/G\longrightarrow H$,
$[i]\rightsquigarrow h_i$, as the degree map.

It remains to check that $<xy,z>=<x,yz>$, for all $x,y,z\in\Lambda$.
For that, it is not restrictive to assume that $x=[a]$, $y=[b]$ and
$z=[c]$, where $a,b,c$ are homogeneous elements in $\bigcup_{i,j\in
I}e_iAe_j$. In such a case, note that if one member of the desired
equality $<xy,z>=<x,yz>$ is nonzero, then $t(x)=i(y)$ and
$t(y)=i(z)$ or, equivalently, $[t(a)]=[i(b)]$ and $[t(b)]=[i(c)]$.
If this holds, then we have

\begin{center}
$<xy,z>=<[a][b],[c]>=<[ab^g],[c]>=(ab^g,c^{g'})$,
\end{center}
where $g,g'\in G$ are the elements such that $g(i(b))=t(a)$ and
$g'(i(c))=g(t(b))$. Note that then $(g^{-1}g')(i(c))=t(b)$ and,
hence,  we also have

\begin{center}
$<x,yz>=<[a],[b][c]>=<[a],[bc^{g^{-1}g'}]>=(a,(bc^{g^{-1}g'})^g)=(a,b^gc^{g'})$.
\end{center}
The equality $<xy,z>=<x,yz>$ follows then from the fact that
$(-,-):A\times A\longrightarrow K$ is a graded Nakayama form for $A$.
\end{proof}

The following is an easy consequence of last proposition
 and its proof. We leave
the proof to the reader.

\begin{cor} \label{cor.preservation Nak-automorphism by covering}
Let $A=\oplus_{h\in H}A_h$ be a weakly basic locally bounded graded pseudo-Frobenius
algebra and let  $(-,-):A\times A\longrightarrow K$  be a
$G$-invariant graded Nakayama form. The following assertions hold:

\begin{enumerate}
\item If $\eta :A\longrightarrow A$ is the  Nakayama
automorphism associated to $(-,-)$, then $\eta\circ g=g\circ\eta$,
for all $g\in G$ \item Let
$<-,->:\Lambda\times\Lambda\longrightarrow K$ be the graded Nakayama
form induced from $(-,-)$ and let
$\bar{\eta}:\Lambda\longrightarrow\Lambda$ be the associated
Nakayama automorphism. Then $\bar{\eta}([a])=[\eta (a)]$ for each
$a\in\bigcup_{i,j}e_iAe_j$.
\end{enumerate}
\end{cor}

The following result completes the last proposition by showing how
to construct $G$-invariant graded Nakayama forms in the split case.

\begin{cor} \label{cor.G-invariant basis and Nakayama form}
Let $A=\oplus_{h\in H}A_h$ be a split   graded pseudo-Frobenius
algebra and let $G$ be a group of graded automorphisms of $A$ which
permute the $e_i$ and acts freely on objects. There exist an element
$\mathbf{h}=(h_i)_{i\in I}\in\prod_{i\in I}Supp(e_iSoc_{gr}(A))$ and
basis $\mathcal{B}_i$ of $e_iA_{h_i}e_{\nu (i)}$, for each $i\in I$, satisfying
the following properties:

\begin{enumerate}
\item $h_i=h_{g(i)}$, for all $i\in I$ \item
$g(\mathcal{B}_i)=\mathcal{B}_{g(i)}$ and $\mathcal{B}_i$ contains an
element of $e_iSoc_{gr}(A)$, for all $i\in I$
\end{enumerate}
 In such case the graded Nakayama form  associated to the pair $(\mathcal{B},\mathbf{h})$
 (see definition \ref{defi.graded Nakayama form associated to basis})
is $G$-invariant.
\end{cor}
\begin{proof}
We fix a subset $I_0\subseteq I$ which is a set of representatives
of the $G$-orbits of objects. Then the assignment
$i\rightsquigarrow [i]$ defines a bijection between $I_0$ and  the
set of objects of $\Lambda=A/G$. For each $i\in I_0$, we fix an $h_i\in
Supp(e_iSoc_{gr}(A))$ and a basis $\mathcal{B}_i$ of $e_iA_{h_i}e_{\nu (i)}$
containing an element $w_i\in e_iSoc_{gr}(A)$, for each $i\in I_0$.
Note that $g(e_iSoc_{gr}(A))=e_{g(i)}Soc_{gr}(A)$  since $G$
consists of graded automorphisms. It then follows that $h_i\in Supp
(e_{g(i)}Soc_{gr}(A))$. Given $j\in I$, the free action of $G$ on
objects implies that there are unique elements $i\in I_0$ and $g\in
G$ such that $g(i)=j$. We then define $h_j=h_i$ and
$\mathcal{B}_j=g(\mathcal{B}_i)$, whenever $j=g(i)$, with $i\in
I_0$. Note that $\mathcal{B}_j$ contains the element $g(w_i)$ of
$e_jSoc_{gr}(A)$. It is now clear that $\mathbf{h}=(h_j)_{j\in I}$
is in $\prod_{j\in I}Supp (e_jSoc_{gr}(A))$ and that $\mathcal{B}_j$
is a basis of $e_jA_{h_j}e_{\nu (j)}$ containing an element of
$e_jSoc_{gr}(A)$, for each $j\in I$. It is also clear that if
$\mathcal{B}:=\bigcup_{j\in I}\mathcal{B}_j$ then
$g(\mathcal{B})=\mathcal{B}$, for all $g\in G$.

By definition of the graded Nakayama form $(-,-):A\times
A\longrightarrow K$ associated to $(\mathcal{B},\mathbf{h})$ (see
definition \ref{defi.graded Nakayama form associated to basis}) and
the fact that $w_j=g(w_j)=w_{g(j)}$, for all $g\in G$ and $j\in I$,
we easily conclude that $(-,-)$ is $G$-invariant.
\end{proof}

We are ready to prove the main result of this section.

\begin{teor} \label{teor.lifting PF in split case}
Let $A$ be a split weakly basic locally bounded graded algebra with enough idempotents and let $G$ be a group of graded automorphisms acting freely on objects. The following assertions are equivalent:

\begin{enumerate}
\item $A$ is pseudo-Frobenius;
\item $\Lambda =A/G$ is pseudo-Frobenius.
\end{enumerate}
\end{teor}
\begin{proof}
$1)\Longrightarrow 2)$ By Corollary \ref{cor.G-invariant basis and Nakayama form}, we have a $G$-invariant graded Nakayama form $(-,-):A\times A\longrightarrow K$. Now Proposition \ref{prop.gr-Frobenius via pushdown functor} gives this implication.

$2)\Longrightarrow 1)$ Let $<-,->:\Lambda\times\Lambda\longrightarrow K$ be a graded Nakayama form for $\Lambda$  and let $\pi :A\longrightarrow A/G=\Lambda$ the canonical functor. We shall define a bilineal form $(-,-):A\times A\longrightarrow K$ as the 'direct sum' of  bilinear forms $(-,-):e_iA_he_j\times e_kA_{h'}e_l\longrightarrow K$, where $h,h'\in H$ and $i,j,k,l\in I$. By definition, this last bilinear form is zero when $j\neq k$ and is the composition

\begin{center}
$(-,-):e_iA_he_j\times e_jA_{h'}e_l\stackrel{\pi\times\pi}{\longrightarrow}e_{[i]}\Lambda_h e_{[j]}\times e_{[j]}\Lambda_{h'}e_{[l]}\stackrel{<-,->}{\longrightarrow}K$,
\end{center}
when $j=k$.
Note that this last bilinear form is zero when $[l]\neq \nu ([i])$ or $h+h'\neq h_{[i]}=\mathbf{h}([i])$, where $\nu$ and $\mathbf{h}$ are the Nakayama permutation and the degree map associated to $<-,->$. We then get that the induced bilinear form $(-,-):e_iAe_j\times e_jAe_l$ is either zero, when $[l]\neq\nu ([i])$, or it is bilinear form of degree $h_{[i]}$ otherwise.  We claim that the resulting bilinear form $(-,-):A\times A\longrightarrow K$ is a graded Nakayama form for $A$, so that $A$ will be graded pseudo-Frobenius. To prove our claim we first check the following two conditions:

\begin{enumerate}
\item[a)] For each $i\in I$, there is a unique $\mu (i)\in I$ such that $(e_iA,Ae_{\mu (i)})\neq 0$;
\item[b)] For each $j\in I$, there is a unique $\mu' (j)\in I$ such that $(e_{\mu'(j)}A,Ae_{j})\neq 0$.
\end{enumerate}
We check a) since b) follows by symmetry. By propositions \ref{prop.uniqueness of Nakayama permutation} and \ref{prop.graded Nak-form via basis}, we know that $e_{[i]}\text{Soc}_{gr}(\Lambda )_he_{\nu([i])}$ is a one-dimensional $K$-vector space whenever $h\in\text{Supp}(e_{[i]}\text{Soc}_{gr}(\Lambda ))$. Moreover, by the $G$-covering condition of the functor $\pi :A\longrightarrow\Lambda$,  we know that we have an isomorphism of $K$-vector spaces $\tilde{\pi}:\oplus_{j\in I,\pi (j)=\nu ([i])}e_iA_he_j\stackrel{\cong}{\longrightarrow} e_{[i]}\Lambda_h e_{\nu ([i])}$. We claim that $\tilde{\pi} (e_i\text{Soc}_{gr}(A)_he_j)\subseteq e_{[i]}\text{Soc}_{gr}(\Lambda )_he_{\nu ([i])}$, for each $j\in I$ such that $\pi (j)=\nu ([i])$. Indeed, if $w\in e_i\text{Soc}_{gr}(A)_he_j=\text{Soc}_{gr}(e_iA)_he_j$, then $w\star 1$ is annihilated on the right by $J\star G=J^{gr}(A\star G)$, so that $w\star 1\in e_i\text{Soc}_{gr}((A\star G)_he_j)$. We now view $w\star 1$ as a morphism $e_j(A\star G)\longrightarrow (A\star G)e_i[h]$ in the category $grproj-A\star G$ of finitely generated projective graded right $A\star G$-module. We then use the fact that $F_\lambda :Gr-A\star G\longrightarrow Gr-\Lambda$ is an equivalence of categories since $\Lambda$ is a weak skeleton of $A\star G$. We then get that $\tilde{\pi} (w\star 1)$ is a morphism $e_{\nu ([i])}\Lambda\longrightarrow e_{[i]}\Lambda$ which is the socle of the category of $grproj -\Lambda$. That is, it is annihilated by any morphism in the radical of this category, which is easily identified with $J^{gr}(\Lambda )$. We then get that $\tilde{\pi}(w\star 1)\in e_{[i]}\text{Soc}_{gr}(\Lambda )e_{\nu ([i])}$, as desired.

The last paragraph shows that we have an induced injective map $\tilde{\pi}:\oplus_{j\in I,\pi (j)=\nu ([i])}e_i\text{Soc}_{gr}(A)_he_j\rightarrowtail e_{[i]}\text{Soc}_{gr}(\Lambda )_he_{\nu ([i])}$, whenever $h\in\text{Supp}(e_{[i]}\text{Soc}_{gr}(\Lambda ))$. We will prove that its domain is nonzero,  which will imply that $\tilde{\pi}$ is bijective since  $\text{dim}(e_{[i]}\text{Soc}_{gr}(\Lambda )_he_{\nu ([i])})=1$ (see Proposition \ref{prop.graded Nak-form via basis}).
It  will follow also that there is exactly one $j\in I$
 such that $e_i\text{Soc}_{gr}(A)_he_j\neq 0$. For our desired proof, we  look at $[i]$ and $\nu ([i])$ as elements of a set of representatives $I_0$ of the $G$-orbits (see the explicit construction of a weak skeleton). The condition that $e_{[i]}\text{Soc}_{gr}(\Lambda )_he_{\nu ([i])}\neq 0$ is then equivalent to the condition that $e_i\text{Soc}_{gr}(A\star G)_he_j\neq 0$, for any $j\in I$ such that $\pi (j)=\nu ([i])$. Take then $0\neq x\in e_{i}\text{Soc}_{gr}(A\star G )_he_{j}$. We then have $x=\sum_{g\in G}x_g\star g$, where $x_g\in e_iA_he_{g(j)}$, for each $g\in G$. Since $x$ is annihilated by $J^{gr}(A\star G)=J\star G$, we immediately get that each $x_g$ is annihilated by $J=J^{gr}(A)$, which implies that $x_g\in e_i\text{Soc}_{gr}(A)_he_{g(j)}$ and, obviously, there is a $g\in G$ such that $x_g\neq 0$.

We now define $h_i:=\tilde{\mathbf{h}}(i)=\mathbf{h} ([i])$, for each $i\in I$, thus obtaining a map $\tilde{h}:I\longrightarrow H$. The last paragraph gives a map $\mu :I\longrightarrow I$ which assigns to $i$ the unique $\mu (i)\in I$ such that $e_i\text{Soc}_{gr}(A)_{h_i}e_{\mu (i)}\neq 0$. Note that, by a symmetric argument, we get a map $\mu':I\longrightarrow I$ which assigns to each $j$ the unique $\mu'(j)\in I$ such that $e_{\mu'(j)}\text{Soc}_{gr}(A)_{h_{\mu'(j)}}e_j\neq 0$. This implies that $\mu$ and $\mu'$ are mutually inverse maps. Note also that, by construction, we have $(ab,c)=(a,bc)$ whenever $a,b,c$ are homogeneous elements in $\bigcup_{i,j\in I}e_iAe_j$. We then have the following chain of double implications:

\begin{center}
$(e_iA,Ae_j)\neq 0$  $\Longleftrightarrow$ there is a $k\in I$ such that $(e_iAe_k,e_kAe_j)\neq 0$ $\Longleftrightarrow$ there are homogeneous elements $a\in e_iAe_k$, $b\in e_kAe_j$ such that $(a,b)\neq 0$ $\Longleftrightarrow$  there are homogeneous elements $a\in e_iAe_k$, $b\in e_kAe_j$ such that $(e_i,ab)\neq 0$ $\Longleftrightarrow$ there is a homogeneous element $w\in e_iAe_j$ such that $0\neq (e_i,w)=<e_{[i]},[w]>$ $\Longleftrightarrow$ there is a homogeneous element $w$ of degree $h_i=\mathbf{h}([i])$ such that $[w]\in e_{[i]}\text{Soc}_{gr}(\Lambda )e_{[j]}$.
\end{center}
We then have $[j]=\nu ([i])$ and the previous paragraphs show that $w\in\text{Soc}_{gr}(A)_{h_i}$ and then necessarily $j=\mu (i)$. Then all conditions of Definition \ref{defi.graded Nakayama form} have been checked for $(-,-):A\times A\longrightarrow K$, except the fact that the induced bilinear form $(-,-):e_iAe_j\times e_jAe_{\mu (i)}\longrightarrow K$ is nondegenerate. But this is  a consequence of the nondegeneracy of $<-,->:e_{[i]}\Lambda e_{[j]}\times e_{[j]}\Lambda e_{\nu ([i])}\longrightarrow K$. Then $(-,-)$ is a graded Nakayama form for $A$.
\end{proof}

 Although the concept of covering of quivers with relations was initially defined for stable translations quivers (see \cite{Ri}, \cite{G2}, \cite{BG}), the most general version can be found in \cite{Gr} and \cite{MVDP}. For us a \emph{covering of a quiver with relations} $F:(Q',\rho ')\longrightarrow (Q,\rho )$ is what is called a morphism of graphs with relations in \cite{Gr}. The important fact for us it that if $F$ is such a covering, then there is a group of automorphisms $G$ of $(Q',\rho ')$ such that if $A'=KQ'/<\rho'>$ and $KQ/<\rho>$ are the associated algebras with enough idempotents, then $G$ is an automorphism of $A'$ which acts freely on objects and we have a covering of  (trivially graded) $K$-categories $A'\stackrel{\pi}{\longrightarrow}A'/G\stackrel{\cong}{\longrightarrow}A$. Moreover, any quiver with relation $(Q,\rho )$ admits a \emph{universal covering} $\tilde{F}:(\tilde{Q},\tilde{\rho})\longrightarrow (Q,\rho )$ in the sense that if $F:(Q',\rho')\longrightarrow (Q,\rho  )$ is any covering of quiver with relations, then there is a unique covering of quiver with relations $F':(\tilde{Q},\tilde{\rho})\longrightarrow (Q',\rho ')$ such that $F\circ F'=\tilde{F}$. In this setting,  it is important to note that if $A=KQ/<\rho >$ is locally finite dimensional, then it is locally bounded
 if, and only if, so is $\tilde{A}=K\tilde{Q}/<\tilde{\rho}>$ (and so is $A'=KQ'/<\rho'>$).

We are now ready to state the following consequence of theorem:

\begin{cor} \label{cor.covering of selfinjective algebra}
Let $(Q,\rho)$ be a quiver with relations such that $A=KQ/<\rho >$ is a self-injective finite dimensional algebra. The following assertions hold:

\begin{enumerate}
\item If $(\tilde{Q},\tilde{\rho})\longrightarrow (Q,\rho )$ is the universal covering, then $\tilde{A}=K\tilde{Q}/<\tilde{\rho}>$ is a pseudo-Frobenius algebra with enough idempotents.
\item If $(Q',\rho')\longrightarrow (Q,\rho)$ is a covering of quivers with relations such that $A'=KQ'/<\rho'>$ is finite dimensional, then $A'$ is a self-injective algebra.
\end{enumerate}
\end{cor}
\begin{proof}
Assertion 1 is a direct consequence of theorem \ref{teor.lifting PF in split case} while assertion 2 follows from this same theorem and from  corollary \ref{cor.Frobenius=simple socle}.
\end{proof}

\section{Acknowledgements}

The authors are supported by research projects of
the Spanish Ministry of Education and Science and the Fundación Seneca
of Murcia, with a part of FEDER funds. They thank both institutions for
their support.

\end{document}